\documentclass{amsart}
\usepackage{amssymb}
\usepackage{latexsym}

\date{}

\newcommand{\Z}{{\mathbb Z}}
\newcommand{\R}{{\mathbb R}}
\newcommand{\C}{{\mathbb C}}
\newcommand{\N}{{\mathbb N}}

\newcommand{\D}{{\mathbb D}}

\newcommand{\Hi}{{\mathcal{H}}}

\DeclareMathOperator*{\esssup}{ess\,sup}

\newtheorem{theorem}{Theorem}[section]
\newtheorem{remark}[theorem]{Remark}
\newtheorem{lemma}[theorem]{Lemma}
\newtheorem{prop}[theorem]{Proposition}
\newtheorem{coro}[theorem]{Corollary}
\newtheorem{defi}[theorem]{Definition}
\renewcommand{\Re}{\mathrm{Re} \, }
\newcommand{\tr}{\mathrm{tr} }
\newcommand{\dd}{\mathrm{d}}

\newcommand{\inn}{\mathrm{in}}
\newcommand{\out}{\mathrm{out}}

\newcommand{\Hd}{\mathrm{H}}
\newcommand{\Pa}{\mathrm{P}}

\title{Dynamics of Unitary Operators}

\author[D.\ Damanik]{David Damanik}

\thanks{D.\ D.\ was supported in part by a Simons Fellowship and NSF grant DMS--1067988.}

\address{Department of Mathematics, Rice University, Houston, TX~77005, USA}

\author[J.\ Fillman]{Jake Fillman}

\thanks{J.\ F.\ was supported in part by NSF grant DMS--1067988.}

\author[R.\ Vance]{Robert Vance}

\begin{document}

\maketitle

\begin{abstract}
We consider the iteration of a unitary operator on a separable Hilbert space and study the spreading rates of the associated discrete-time dynamical system relative to a given orthonormal basis. We prove lower bounds for the transport exponents, which measure the time-averaged spreading on a power-law scale, in terms of dimensional properties of the spectral measure associated with the unitary operator and the initial state. These results are the unitary analog of results established in recent years for the dynamics of the Schr\"odinger equation, which is a continuum-time dynamical system associated with a self-adjoint operator. We discuss how these general results may be studied by means of subordinacy theory in cases where the unitary operator is given by a CMV matrix. An example of particular interest in which this scenario arises is given by a time-homogeneous quantum walk on the integers. For the particular case of the time-homogeneous Fibonacci quantum walk, we illustrate how these components work together and produce explicit lower bounds for the transport exponents associated with this model.
\newline
\end{abstract}

\newline


\section{Introduction}

This paper is concerned with the dynamics of unitary operators acting on Hilbert spaces.  Specifically, we  fix a (separable) Hilbert space $\Hi$, an orthonormal basis $(\varphi_n)_{n \in A}$ for $\Hi$,  a unitary operator $U : \Hi \to \Hi$, and a unit vector $\psi \in \Hi$ and consider the discrete-time evolution $\psi(k) = U^k \psi$. The orthonormal basis $(\varphi_n)_{n \in A}$ will be indexed by a suitable countable set $A$ - in general we may always take $A = \Z_+$, but other countable sets may be more natural in certain settings.  For example, in the case $\Hi = \ell^2(\Z^d)$, it is natural to use the orthonormal basis $(\delta_n)_{n \in \Z^d}$.  Our goal is to give as complete a dynamical picture as possible for the spreading of $U^k \psi$ with respect to the basis $(\varphi_n)_{n \in A}$ in terms of spectral characteristics of $U$.  By the spectral theorem, there is a Borel probability measure $\mu_{\psi}^U $ on the circle $\partial \D = \{ z \in \C : |z| = 1 \}$ such that
\[
\left\langle \psi, f(U) \psi \right\rangle = \int_{\sigma(U)} \! f(z) \, \dd \mu_{\psi}^U(z)
\]
for any bounded, Borel measurable function $f$ on $\partial \D$.  Typically, we will suppress the dependence of $\mu_{\psi}^U$ on $U$ and simply write $\mu_\psi$.

There is an extensive literature devoted to the ``self-adjoint case,'' that is, the time-evolution associated with the Schr\"odinger equation. Given a self-adjoint operator $H$ in $\Hi$ and a unit vector $\psi \in \Hi$, one studies the spreading of $e^{-itH} \psi$ relative to a given orthonormal basis. We refer the reader to \cite{DT, G1, G2, GSB, L} and the references therein. Despite the obvious analogies, much less is known about the unitary case -- see \cite{BGVW, GVWW, OS}, for example. One may be tempted to reduce questions about the unitary case to known results in the self-adjoint case. In our opinion this has at least two drawbacks. The self-adjoint case leads naturally to Ces\`aro averages in the continuous time-parameter and hence such results have no meaning in the unitary setting, where the time-parameter is discrete. Moreover, a problem in the unitary case is often given by an explicit unitary operator and one wants to take advantage of the (usually) simple structure of the operator, which may not be present in any associated self-adjoint operator. Specifically, we will discuss the class of CMV matrices later in the paper and there is an extensive set of tools one can use to analyze such an operator and verify the input to the general dynamical results in the unitary case. It will be obvious that it is extremely desirable to stay within the class of CMV matrices when proving delicate spectral properties such as $\alpha$-continuity, which are often difficult to establish for a given operator.

As a consequence, we prefer to work out the analogies between the self-adjoint case and the unitary case. That is, we will establish general results in the unitary case that mirror known results in the self-adjoint case. The advantage of this approach is that we obtain results that can readily be applied to a given setting in which a time-evolution is given by the iteration of a unitary operator.

For example, there has been a lot of activity recently in the study of quantum walks; compare \cite{CGMV, CGMV2, J11, J12, JM} and references therein. A quantum walk is indeed given by the iteration of a unitary operator, and hence our general results below apply directly to any quantum walk. More specifically, a time-homogeneous quantum walk on the integers can be related in a simple way to a CMV matrix, as pointed out in \cite{CGMV}. Thus, for quantum walks of this kind, one wants to take advantage of the tools available that allow one to prove the required spectral continuity results for spectral measures. One of the primary tools here is subordinacy theory, which derives spectral continuity from solution estimates. The latter may be obtained by an analysis of the transfer matrices associated with a given CMV matrix.

The structure of the paper is as follows. In Section~\ref{s.2} we introduce some basic quantities that capture the aspects of the dynamics in which we are interested. In particular, the transport exponents are defined. In Section~\ref{s.3} we establish lower bounds for the transport exponents in terms of regularity of the spectral measure. In particular, the Hausdorff dimension and the packing dimension of the spectral measure play a crucial role. In Section~\ref{s.4} we focus on the special case where the unitary operator in question is given by a CMV matrix. We discuss how subordinacy theory provides an elegant way of establishing the spectral regularity that was shown to imply lower transport bounds. Then we turn to the connection between quantum walks on the integers and CMV matrices and explain how the material from earlier parts of the paper gives a useful framework in which the spreading rates of a quantum walk on the integers may be studied. Finally, we use the Fibonacci quantum walk as an example for which we implement this overall strategy, and derive explicit lower bounds for the spreading rates associated with this model from the connection with CMV matrices, an analysis of the solutions which provides the input to subordinacy theory and hence implies spectral regularity, and the derivation of the lower bounds for the transport exponents from these spectral regularity properties.

\bigskip

\noindent\textbf{Acknowledgements.} We are grateful to the anonymous referees for several useful suggestions that led to marked improvements of the paper.

\section{Preliminaries and Basic Definitions}\label{s.2}

We shall mirror the notation and development found in \cite{DT}.  Let $\Hi$ be a complex separable Hilbert space, $U$ a unitary operator on $\Hi$, and $\psi \in \Hi$ such that $\| \psi\| = 1$.  We are interested in the time evolution of the vector $\psi$, that is, $\psi(k) = U^k \psi$.  Let $(\varphi_n)_{n \in A}$ be an orthonormal basis for $\Hi$, indexed by a suitable countable set $A$ -- in this paper, we will consider $A=\Z_+^d,\Z^d$ as appropriate.  To describe the spreading of $\psi$ with respect to the basis $(\varphi_n)_{n \in A}$, we first define
\[ a_{\psi}(n,k) = \left| \left\langle \varphi_n , \psi(k) \right\rangle \right|^2, \]
which can be thought of as the probability that $\psi$ is in the state $\varphi_n$ at time $k$. We shall also be interested in the Ces\`aro time-averaged probabilities, given by
\[
\tilde{a}_{\psi}(n,K) = \frac{1}{K} \sum_{k=0}^{K-1} a_{\psi}(n,k).
\]

Throughout the paper, we shall be interested in Ces\`aro averages of quantities, so, for a function $f: \Z_{\geq 0} \to \R$, we introduce the notation $\langle f \rangle$ to denote the average of $f$.  More precisely, we set
\[
\langle f \rangle(K) = \frac{1}{K} \sum_{j=0}^{K-1} f(j).
\]
For example, in this notation, one could write $\tilde{a}_{\psi}(n,K) = \langle a_{\psi}(n,\cdot) \rangle(K)$.

For fixed $k,K$, a straightforward computation reveals
\[ \sum_{n \in A} a_{\psi}(n,k) = \sum_{n \in A} \tilde{a}_{\psi}(n,K) = 1 \]
since $(\varphi_n)_{n \in A}$ is an orthonormal basis for $\Hi$.

Given $R \geq 0$, we are interested in the probability of finding $\psi$ within a ball of radius $R$. Specifically, we define
\begin{align*}
P_{\inn}^{\psi}(R,k) & = \sum_{|n| \leq R} a_{\psi}(n,k), \\
P_{\out}^{\psi}(R,k) & = \sum_{|n| > R} a_{\psi}(n,k) = 1 - P_{\inn}^{\psi}(R,k),
\end{align*}
and their time-averaged counterparts
\begin{align*}
\tilde{P}_{\inn}^{\psi}(R,K) & = \sum_{|n| \leq R} \tilde{a}_{\psi}(n,K) = \langle P_{\inn}(R,\cdot) \rangle(K) , \\
\tilde{P}_{\out}^{\psi}(R,K) & = \sum_{|n| > R} \tilde{a}_{\psi}(n,K).
\end{align*}
In the above formulae, $|n|$ denotes the $\ell^1$ norm of $n$, that is, $|n| = |n_1| +\cdots + |n_d|$.

We shall also describe transport behavior of $U$ and $\psi$ in terms of the moments of the position operator, defined by
\[ |X|^p_{\psi}(k) = \sum_n \left( |n|^p +1 \right)a_{\psi}(n,k), \]
with the time-averaged counterparts
\[ \left\langle |X|^p_{\psi} \right\rangle(K) = \sum_n \left( |n|^p +1 \right) \tilde{a}_{\psi}(n,K) = \frac{1}{K} \sum_{k=0}^{K-1} |X|^p_{\psi}(k). \]

\begin{remark} \label{r.moments.pout}
It is helpful to observe that
$$
\left\langle |X|^p_\psi \right\rangle(K) \geq R^p \tilde{P}_{\out}(R,K)
$$
for all $R,K$
\end{remark}

We would like to compare the growth of $|X|^p_{\psi}(k)$ to polynomial growth of the form $k^{\beta p}$ for a suitable exponent $\beta$.  In light of this, the following transport exponents are natural objects to consider
\begin{align*}
\beta_{\psi}^+(p) & = \limsup_{k \to \infty} \frac{\log \left( |X|^p_{\psi}(k) \right)}{p \log(k)}, \\
\beta_{\psi}^-(p) & = \liminf_{k \to \infty} \frac{\log \left( |X|^p_{\psi}(k) \right)}{p \log(k)}, \\
\tilde{\beta}_{\psi}^+(p) & = \limsup_{K \to \infty} \frac{\log \left( \left\langle |X|^p_{\psi} \right\rangle (K) \right)}{p \log(K)}, \\
\tilde{\beta}_{\psi}^-(p) & = \liminf_{K \to \infty} \frac{\log \left( \left\langle|X|^p_{\psi}\right\rangle(K) \right)}{p \log(K)}.
\end{align*}

By Jensen's inequality, the functions $\beta^{\pm}_{\psi}$ and $\tilde{\beta}^{\pm}_{\psi}$ are non-decreasing functions of $p$.  For a detailed proof of this, the interested reader may consult Lemma 2.7 of \cite{DT}, for example.

Usually, the initial state $\psi$ will be explicitly given or clear from context, so we will often suppress the dependence of the dynamical quantities on $\psi$, and simply write $a(n,k), P_{\inn}(R,k), |X|^p(k)$, etc.

\section{Transport and Singular Continuous Spectrum}\label{s.3}

In this section we prove estimates for the dynamical quantities introduced in the previous section that hold for general unitary operators. We begin with results that rely on suitable regularity properties of the spectral measure associated with the operator $U$ and the initial state $\psi$. Specifically, we first consider the case where the measure is uniformly $\alpha$-H\"older continuous for some $\alpha > 0$ and then study the case of measures that have a non-trivial $\alpha$-continuous component, that is, measures that are not singular with respect to $\alpha$-dimensional Hausdorff measure. In fact, the latter case can be understood by approximation with measures covered by results in the former case. As a consequence, we obtain quantitative estimates in terms of the most continuous component of the spectral measure. One should emphasize that these estimates are strictly one-sided. That is, based on the analogy to the self-adjoint case, one may expect that in some cases, transport can be fast even if the spectral measure is highly singular.

\subsection{Uniformly $\alpha$-H\"older Continuous Spectral Measures}

We shall first be interested in a description of continuity and singularity of measures supported on  $\partial \D$. To that end, let us first recall the notion of uniform H\"{o}lder continuity for such measures.

\begin{defi}
We will say that a measure $\mu$ on $\partial \D$ is uniformly $\alpha$-H\"older continuous {\rm (}U$\alpha$H{\rm )} if there exists a constant $C>0$ such that for every arc $I \subseteq \partial \D$, we have $\mu(I) < C|I|^{\alpha}$, where $| \cdot |$ shall be taken to mean one-dimensional Lebesgue measure on $\partial \D$.
\end{defi}

We remark that a measure $\mu$ on $\partial \D$ which is U$\alpha$H must necessarily be finite, for $\mu(\partial \D) \leq C |\partial \D|^{\alpha} < \infty$.
\newline

The following lemma provides the critical estimate for the results that follow.

\begin{lemma}  \label{UaH.integral.estimate}
Suppose $\mu$ is a  U$\alpha$H measure on $\partial \D$ for $0 \leq \alpha < 1$. There exists a constant $\gamma > 0$ such that
\[ \int_{\partial \D} \left| \frac{\overline{z}^Kw^K - 1}{\overline{z} w - 1 } \right| \, \mathrm d \mu(w) \leq \gamma K^{1-\alpha} \]
for all $z \in \partial \D$ and all $K \in \Z_+$. In particular, $\gamma$ depends on neither $z$ nor $K$.
\end{lemma}

\begin{proof}
By uniformity of $\mu$, it is no loss of generality to assume that $z = 1$, so we may consider the integral
\[ \int_{\partial \D} \left| \frac{w^K - 1}{w - 1 } \right| \, \mathrm d \mu(w). \]

The case $\alpha = 0$ is trivial: take $\gamma = \mu(\partial \D)$ and observe that
\[
\left| \frac{w^K-1}{w-1} \right| = \left| 1+w+w^2 + \cdots + w^{K-1} \right| \leq K
\]
for all $w \in \partial \D$ by the triangle inequality, so that the integral in question is bounded by $\mu(\partial \D) \cdot K$.
\newline

Next, suppose $0 < \alpha < 1$.  For each $K$, there are three parts of the integral that we will control:
\begin{align*}
S_1 & = \{ z \in \partial \D : \Re \! (z) \leq 0 \}, \\
S_2 & = \left\{ e^{i \theta} : -\frac{\pi}{2K} \leq \theta \leq \frac{\pi}{2K}  \right\}, \\
S_3 & = \left\{ e^{i \theta} : \frac{\pi}{2K} < \theta \leq \frac{\pi}{2}  \textup{ or } -\frac{\pi}{2} \leq \theta < -\frac{\pi}{2K}\right\}.
\end{align*}
It is easy to see that $\left| \frac{w^K - 1}{w - 1} \right| \leq \sqrt{2}$ on $S_1$, and hence
\[
\int_{S_1} \! \left| \frac{w^K - 1}{w - 1} \right| \,  \mathrm d \mu(w) \leq \sqrt{2} \mu(\partial \D)
\]
Since $\mu$ is U$\alpha$H, choose $C$ such that $\mu(I) \leq C|I|^{\alpha}$ for arcs $I$.  In particular,
\[
\mu(S_2) \leq C\left( \frac{\pi }{K} \right)^{\alpha} = C_1 K^{-\alpha}
\]
with $C_1 =  C \pi^{\alpha} $.
\newline

Since
$
\left| \frac{w^K-1}{w-1} \right|  \leq K
$
for all $w \in S_2$ (indeed for all $w \in \partial \D$), we have
\begin{align*}
\int_{S_2} \! \left| \frac{w^K - 1}{w - 1} \right| \, \mathrm d \mu(w) \leq K \mu(S_2) \leq C_1 K^{1-\alpha}.
\end{align*}

Lastly, we consider $S_3$. We can decompose
\[
S_3 \subseteq \bigcup_{l = 1}^{\left\lfloor \sqrt{K} \right\rfloor} (A_l \cup B_l)
\]
with $A_l = \left\{ e^{i \theta} : \frac{l^2 \pi}{2K} < \theta \leq \frac{(l+1)^2 \pi}{ 2K} \right\} $ and $B_l = \left\{ e^{i \theta} : -\frac{(l+1)^2 \pi}{2K} \leq \theta < -\frac{l^2 \pi}{ 2K} \right\} $. We have
\[
\mu(A_l), \mu(B_l) \leq C \left( \frac{(l+1)^2\pi}{2 K} - \frac{l^2 \pi}{2 K} \right)^{\alpha} \leq C_2 l^{\alpha} K^{-\alpha}
\]
with $C_2 = C \left( \frac{3\pi}{2} \right)^{\alpha}$

Additionally, for $w \in A_l \cup B_l$, one has $\left| w^K - 1 \right| \leq 2$ and
\[
 |w-1| \geq \sin \left( \frac{l^2 \pi}{2K} \right) \geq \frac{l^2}{K}.
\]
Thus, if $\alpha < 1$, we have
\begin{align*}
\int_{S_3} \! \left| \frac{w^K - 1}{w - 1} \right| \, \mathrm d \mu(w)
       & \leq \sum_{l=1}^{ \left\lfloor \sqrt{K} \right\rfloor } \int_{A_l} \! \left| \frac{w^K - 1}{w - 1} \right| \, \mathrm d \mu(w)  + \sum_{l=1}^{\left\lfloor \sqrt{K} \right\rfloor} \int_{B_l} \! \left| \frac{w^K - 1}{w - 1} \right| \, \mathrm d \mu(w)\\
       & \leq \sum_{l = 1}^{\left\lfloor \sqrt{K} \right\rfloor} \frac{4K}{l^2} \cdot C_2 l^{\alpha} K^{-\alpha}  \\
       & \leq 4 C_2 K^{1 - \alpha} \sum_{l=1}^{\infty} l^{\alpha - 2} \\
       & \leq C_3 K^{1-\alpha}
\end{align*}
with $C_3 = 4 C_2 \sum_{l=1}^{\infty} l^{\alpha - 2}$ (note that the series converges because $\alpha < 1$).  Combining the three estimates for $S_1, S_2,$ and $S_3$ gives us the desired bound.
\end{proof}

\begin{remark}
It is relatively easy to see that the proof of Lemma \ref{UaH.integral.estimate} yields an upper bound of constant times $\log(K)$ in the case when $\alpha = 1$.  Moreover, it is well-known and not hard to verify that this upper bound is optimal when $\mu$ is Lebesgue measure on $\partial \D$.  However, the factor of $\log(K)$ would not be optimal in the following lemma, which is why a separate argument is necessary therein.
\end{remark}

Lemma~\ref{UaH.integral.estimate} gives estimates for Fourier coefficients on the unit circle very similar to those in Strichartz' theorem \cite{S}.

\begin{lemma} \label{strichartz.estimates}
Suppose $\mu$ is a U$\alpha$H measure on $\partial \D$ for some $0 \leq \alpha \leq 1$. For each $f \in L^2(\partial \D, \mathrm d \mu)$, $k \in \Z$, define
\[ \widehat{f \mu}(k) = \int_{\partial \D} \! z^{-k} f(z) \, \mathrm d \mu(z). \]
Then, there exists $C>0$ such that for all $f \in L^2(\partial \D, \mathrm d \mu)$ and $K>0$, we have
\[
\left\langle \left| \widehat{f\mu} \right|^2 \right\rangle(K) < C \|f\|_{L^2(\mu)}^2 K^{-\alpha}.
\]
\end{lemma}

\begin{proof}
First, suppose $\alpha < 1$.  A straightforward calculation reveals
\begin{align*}
\left\langle \left| \widehat{f\mu} \right|^2 \right\rangle(K) & = \frac{1}{K} \sum_{j=0}^{K-1} \left| \widehat{f \mu}(j) \right|^2 \\
                      & = \frac{1}{K} \sum_{j=0}^{K-1}  \int_{\partial \D} \int_{\partial \D} \! \overline{z}^j w^j f(z) \overline{f(w)} \, \mathrm d \mu(z) \, \mathrm d \mu(w) \\
                      & = \frac{1}{K} \int_{\partial \D} \int_{\partial \D} \! \frac{\overline{z}^K w^K - 1}{\overline{z} w - 1 } f(z) \overline{f(w)} \, \mathrm d \mu(z) \, \mathrm d \mu(w).
\end{align*}
By using the elementary inequality $|ab| \leq \frac{1}{2} |a|^2 + \frac{1}{2} |b|^2$ and Fubini's Theorem, we see that this is in turn bounded above by
$$
                      \frac{1}{K}  \int_{\partial \D} \int_{\partial \D} \! \left| \frac{\overline{z}^K w^K - 1}{\overline{z} w - 1 } \right| \left|f(z) \right|^2  \, \mathrm d \mu(z) \, \mathrm d \mu(w).
$$
Now, we integrate with respect to $w$ and apply the previous lemma to see that this is less than or equal to
$$
                      \frac{1}{K} \gamma K^{1-\alpha}  \int_{\partial \D} \! \left| f(z) \right|^2 \, \mathrm d \mu(z)
                       = \gamma K^{-\alpha} \|f\|^2.
  $$

The case $\alpha = 1$ is essentially elementary.  Suppose $\mu$ is U1H, and let $\lambda$ denote (normalized) Lebesgue measure on $\partial \D$. Evidently, $\mu$ is absolutely continuous with respect to $\lambda$. Moreover, by the Lebesgue Differentiation Theorem, $g = \frac{d\mu}{d \lambda} $ is in $L^\infty(\lambda)$.  Notice that $f\sqrt{g} \in L^2(\lambda)$  and, evidently, $ \left\| f \sqrt{g} \right\|_{L^2(\lambda)} = \| f \|_{L^2(\mu)} $.  Thus, by Plancherel, we have
\begin{align*}
\left\langle \left| \widehat{f\mu} \right|^2 \right\rangle(K)
& =
\frac{1}{K} \sum_{j=0}^{K-1} \left| \widehat{f\mu}(j) \right|^2 \\
& =
\frac{1}{K} \sum_{j=0}^{K-1} \left| \widehat{fg}(j) \right|^2 \\
& \leq
\frac{1}{K} \left\| f g \right\|_{L^2(\lambda)}^2 \\
& \leq
\frac{\| g \|_\infty}{K} \left\| f \sqrt{g} \right\|_{L^2(\lambda)}^2 \\
& =
\frac{\| g \|_\infty}{K} \|f\|_{L^2(\mu)}^2.
\end{align*}
Notice that $\widehat{\cdot}$ has two different meanings in the above argument.  In the first line, it is as defined in the statement of the lemma, while, in the second line, it denotes the usual Fourier transform $L^2(\partial \D, \lambda) \to \ell^2(\Z)$.
\end{proof}

\begin{prop} \label{UaH}
Let $\Hi$, $U$, $\psi$, and $(\varphi_n)_{n \in A}$ with $ A = \Z_+^d$ or $A = \Z^d$.  If the spectral measure $\mu_{\psi}$ is U$\alpha$H for some $0 \leq \alpha \leq 1$,  then there is a uniform constant $C_0 > 0$ such that the following hold for all $N,K \geq 1$:
\[
\tilde{P}_{\inn}(N,K) \leq C_0 N^d K^{-\alpha}.
\]
As a consequence, for each $p > 0$, there exists a constant $C_p > 0$ such that the following holds for all $K$:
\[
\left\langle |X|_{\psi}^p \right\rangle(K) \geq C_p K^{\frac{p\alpha}{d}}.
 \]
\end{prop}

\begin{proof}
 Let $\Hi_{\psi}$ denote the cyclic subspace spanned by $U$ and $\psi$, with the corresponding orthogonal projection $P_{\psi}: \Hi \to \Hi_{\psi}$. Next, let $V: \Hi_{\psi} \to L^2(\partial \D, \mathrm d \mu_{\psi}(z))$ denote the natural unitary equivalence sending $f(U) \psi$ to $f(z)$.  Put $u_{\psi}^n = VP_{\psi} \varphi_n$. We may observe that
$$
a(n,k) = \left| \widehat{u_\psi^n \mu_\psi}(k) \right|^2
$$
by Fubini's theorem and the spectral theorem.  Hence, we obtain
\begin{align*}
\tilde{P}_{\inn}(N,K) & = \frac{1}{K} \sum_{k=0}^{K-1} \sum_{|n| \leq N} a(n,k) \\
                      & = \sum_{|n| \leq N} \left\langle \left| \widehat{u_{\psi}^n \mu_{\psi}} \right|^2 \right\rangle(K) \\
                      & \leq \sum_{|n| \leq N} C K^{-\alpha} \|u_{\psi}^n\|^2_{L^2(\partial \D,\mathrm d \mu_{\psi})} \\
                      & \leq C_0 K^{-\alpha} N^d.
\end{align*}
The third line follows from Lemma \ref{strichartz.estimates}, and the final line follows from the observations $\|u_{\psi}^n\|_{L^2(\partial \D,\mathrm d \mu_{\psi})} = \|P_{\psi} \varphi_n\| \leq \| \varphi_n \| = 1$ and $\#\{n : |n| \leq N\} \sim N^d$.

This implies that
\[
\tilde{P}_{\inn}  \left(
\left( \frac{K^{\alpha}}{2C_0 } \right)^{1/d}, K \right) \leq \frac{1}{2}.
\]
Equivalently,
\[
\tilde{P}_{\out}  \left(
\left( \frac{K^{\alpha}}{2C_0 } \right)^{1/d}, K \right) \geq \frac{1}{2}.
\]
As a consequence of Remark~\ref{r.moments.pout}, we then obtain the estimate
\[
\left\langle |X|^p_{\psi} \right\rangle(K) \geq \left( \frac{K^{\alpha}}{2C_0 } \right)^{p/d} \cdot \frac{1}{2}
\]
With $C_p = \frac{1}{2 (2C_0)^{p/d}} $, we obtain the desired lower bound.
\end{proof}

We can use the previous proposition to prove a reformulation of Theorem 3.2 of \cite{L} in the present context:

\begin{prop} \label{rage}
Suppose that $\mu_{\psi}$ is U$\alpha$H for some $0 \leq \alpha \leq 1$. There then exists a constant $C = C_{\psi}$  such that the following holds for any compact operator $A$, any $p \in \N$, and any $K > 0$:
\[
\frac{1}{K} \sum_{j=0}^{K-1} |\langle \psi(k), A\psi(k) \rangle | < C_{\psi}^{1/p} \| A \|_{p}  K^{-\alpha/p}.
\]
The expression $\| A \|_p$ denotes the $p$th trace norm of $A$, that is,
\[
\| A \|_p = \left( \tr \left( |A|^p \right) \right)^{1/p}.
\]
We allow the possibility that $\| A \|_p = \infty$, in which case the conclusion of the theorem is trivial.
\end{prop}

\begin{proof}
The proof is essentially identical to that given by Last for the self-adjoint case in \cite{L}.  We provide the details for the case $ p>1$ for the convenience of the reader.  The result when $p = 1$ is significantly easier.
\newline

As before, let $P_{\psi}: \Hi \to \Hi_{\psi}$ denote the orthogonal projection onto the cyclic subspace spanned by $U$ and $\psi$, and $V : \Hi_{\psi} \to L^2(\partial \D, \mathrm d \mu_{\psi})$ the standard unitary equivalence.  Given $\phi \in \Hi$ with $\| \phi \| = 1$, put $f_{\phi} = V P_{\psi} \phi$.  Evidently, $\|f_{\phi}\|_{L^2(\partial \D,\mathrm d\mu_{\psi})} \leq 1$.  We may observe that
\begin{align*}
|\langle \phi, \psi(k) \rangle|
  & = |\langle \phi, U^k \psi \rangle| \\
  & = |\langle f_{\phi}(z), z^k \rangle_{L^2(\partial \D, \mathrm d \mu_{\psi})}| \\
  & = \left| \widehat{f_{\phi} \mu_{\psi}}(k) \right|.
\end{align*}
In particular, Lemma~\ref{strichartz.estimates} implies that there exists a constant $C_{\psi}$ (which does not depend on $\phi$)  such that
\begin{equation} \label{e.expect.ub}
 \left\langle | \langle \phi, \psi(\cdot) \rangle |^2 \right\rangle(K) \leq C_{\psi} K^{-\alpha}.
\end{equation}
By the singular value decomposition, there exist real numbers $s_n \geq 0$ and orthonormal bases $(x_n)$ and $(y_n)$ of $\Hi$ such that $A$ can be written as
\[ A\phi = \sum_{n} s_n \langle x_n, \phi \rangle y_n. \]
Moreover, it is well known that $\|A\|_p = \left( \sum_n s_n^p \right)^{1/p}$. Let $q \in (1, \infty )$ denote the exponent conjugate to $p$, so that $\frac{1}{p} + \frac{1}{q} =1$.  We may observe that, for each $k$, one has
\begin{align*}
| \langle \psi(k) , A \psi(k) \rangle  |
   & = \left| \left\langle \psi(k), \sum_n s_n \langle x_n, \psi(k) \rangle y_n   \right\rangle \right| \\
   & \leq \sum_n s_n |\langle x_n, \psi(k) \rangle \langle y_n, \psi(k) \rangle |.
\end{align*}

Thus, we obtain
$$
\frac{1}{K} \sum_{k=0}^{K-1} |\langle \psi(k), A\psi(k) \rangle |
    \leq \frac{1}{K} \sum_n s_n \sum_{k=0}^{K-1} |\langle x_n, \psi(k) \rangle \langle y_n, \psi(k) \rangle |.
$$
Applying Cauchy-Schwarz to the summation over $k$, we see that the expression on the right hand side is bounded above by
$$
    \frac{1}{K} \sum_n s_n \left(     \sum_{k=0}^{K-1} |\langle x_n, \psi(k) \rangle|^2 \right)^{1/2}
 \left( \sum_{k=0}^{K-1} |\langle y_n, \psi(k) \rangle| ^2 \right)^{1/2}.
$$
Applying H\"older's inequality to the summation over $n$, this is in turn bounded above by
$$
 \frac{1}{K} \left( \sum_n s_n^p \right) ^{1/p} \left( \sum_n \left(     \sum_{k=0}^{K-1} |\langle x_n, \psi(k) \rangle|^2 \right)^{q/2}
 \left( \sum_{k=0}^{K-1} |\langle y_n, \psi(k) \rangle| ^2 \right)^{q/2} \right)^{1/q}.
 $$
 Using Cauchy-Schwarz one last time on the right-hand summation in $n$ and the definition of the $p$th trace norm, this expression is less than or equal to
 $$
 \|A\|_p \left( \left( \sum_n \left( \frac{1}{K} \sum_{k=0}^{K-1}|\langle x_n,\psi(k) \rangle|^2 \right)^q \right) \left( \sum_n \left( \frac{1}{K} \sum_{k=0}^{K-1}|\langle y_n,\psi(k) \rangle|^2 \right)^q \right) \right)^{\frac{1}{2q}}.
 $$
Since $(x_n)$ and $(y_n)$ are orthonormal bases of $\Hi$, unitarity of $U$  and \eqref{e.expect.ub} imply that the expression above is in turn bounded above by
$$
    \|A\|_p  \left( \left( C_{\psi} K^{-\alpha} \right)^{q-1} \left( C_{\psi} K^{-\alpha} \right)^{q-1}  \right)^{\frac{1}{2q}}
    =  \|A\|_p C_{\psi}^{1/p} K^{-\alpha/p},
$$
where we have used $p^{-1} + q^{-1} = 1$.

\end{proof}

\subsection{Spectral Measures with a Non-Trivial $\alpha$-Continuous Component}

The results above can be strengthened further.  We briefly review the definition of $\alpha$-dimensional Hausdorff measure.

\begin{defi}
Fix $\alpha \geq 0$, and let $E \subseteq \partial \D$.  Given $\delta >0$, by a $\delta$-cover of $E$, we shall mean a {\rm (}countable{\rm )} collection of subsets $S_1,S_2,\ldots  \subseteq \partial \D$ which satisfies $\textup{diam}(S_n) < \delta$ for all $n$ and $E \subseteq \bigcup_{n=1}^{\infty} S_n$.  The collection of all $\delta$-covers of $E$ will be denoted $\mathcal I_{\delta}(E)$.  The $\alpha$-dimensional Hausdorff measure of $E$ is then defined by
\[
h^{\alpha}(E) = \lim_{\delta \to 0^+} \, \inf_{(S_n) \in \mathcal I_{\delta}(E) } \sum_{n=1}^{\infty} \left( \textup{diam}(S_n) \right)^{\alpha}.
\]
\end{defi}

Note that the infimum of such sums over $\delta$-covers is monotone in $\delta$ so that the indicated limit indeed exists.

The Hausdorff dimension of a non-empty subset $S$ of $\partial \D$ is given by
\begin{align*}
\dim_{\Hd}(S) & = \sup \{ \alpha : h^\alpha(S) > 0 \} = \sup \{ \alpha : h^\alpha(S) = \infty \} \\ & = \inf \{ \alpha : h^\alpha(S) < \infty \} = \inf \{ \alpha : h^\alpha(S) = 0 \}.
\end{align*}

We shall say that a measure $\mu$ on $\partial \D$ is $\alpha$-continuous ($\alpha$c) if $\mu(E) = 0$ for all sets $E \subseteq \partial \D$ having $h^{\alpha}(E) = 0$.  One can easily check that a U$\alpha$H measure on $\partial \D$ must necessarily be $\alpha$-continuous.  The converse need not hold in general, but we can adapt a theorem of Rogers and Taylor for measures on $\R$ to see that an $\alpha$-continuous measure on $\partial \D$ is ``almost'' a U$\alpha$H measure.  The precise formulation follows.

\begin{lemma} \label{rogers.taylor}
Suppose $\mu$ is a finite $\alpha$-continuous measure on $\partial \D$.  Then, for each $\epsilon >0$, there exist mutually singular Borel measures $\mu_1^{\epsilon}$ and $\mu_2^{\epsilon}$ on $\partial \D$ such that $\mu = \mu_1^{\epsilon} + \mu_2^{\epsilon}$, $\mu_1^{\epsilon}$ is U$\alpha$H, and $\mu_2^{\epsilon}(\partial \D) < \epsilon$.
\end{lemma}

\begin{proof}
Let $T: e^{i\theta} \mapsto \theta$ be the usual map from $\partial \D$ to $[0,2\pi)$.  Evidently, $\nu := T_*\mu$ is an $\alpha$-continuous measure on $[0,2\pi) \subset \R$, and hence, we may invoke the result of Rogers and Taylor for measures on $\R$ \cite{RT1, RT2} to produce mutually singular Borel measures $\nu_1^{\epsilon},\nu_2^{\epsilon}$ on $[0,2\pi)$ such that $\nu = \nu_1^{\epsilon} + \nu_2^{\epsilon}$, $\nu_1^{\epsilon}$ is U$\alpha$H, and $\nu_2^{\epsilon}([0,2\pi)) < \epsilon$.  Some slight untangling shows that $\mu_1^{\epsilon} := \left(T^{-1} \right)_* \nu_1^{\epsilon}$ and $\mu_2^{\epsilon} := \left(T^{-1} \right)_* \nu_2^{\epsilon}$ are the desired measures on $\partial \D$.
\end{proof}

We shall say that $\mu$ is $\alpha$-singular if it is supported on a set having zero $\alpha$-dimensional Hausdorff measure.  This leads to a natural decomposition of our Hilbert space, $\Hi = \Hi_{\alpha c} \oplus \Hi_{\alpha s}$, where $\Hi_{\bullet} = \{ \psi \in \Hi : \mu_{\psi} \textup{ is }  \bullet \}$.  One can check that these are closed, mutually orthogonal subspaces of $\Hi$.  As usual, let us denote by $P_{\bullet}$ the orthogonal projection onto $\Hi_{\bullet}$.

Theorem~6.1 of \cite{L} generalizes to the present context.

\begin{prop} \label{ac.not.zero}
Suppose that $P_{\alpha c}\psi \neq 0$ for some $\alpha \in [0,1]$.    Choose $d$ as in Proposition \ref{UaH}. Then, for each $p > 0$, there exists a constant $C = C_{\psi,p}$ such that for every $K > 0$, one has
\[
\big\langle  |X|^p_{\psi}  \big\rangle(K) > C_{\psi,p} K^{\alpha p/d}.
\]
\end{prop}

\begin{proof}
The proof of this result is again essentially identical to that of Theorem~6.1 in \cite{L}.  We provide the details for completeness.
\newline

Put $\psi_{\alpha c} = P_{\alpha c} \psi$ and $\psi_{\alpha s} = P_{\alpha s}\psi = \psi - \psi_{\alpha c}$. By Lemma~\ref{rogers.taylor}, we may choose Borel measures $\mu_1$ and $\mu_2$ on $\partial \D$ such that $\mu_{\psi_{\alpha c}} = \mu_1 + \mu_2$, $\mu_1$ is U$\alpha$H, $\mu_2(\partial \D) < \frac{1}{2} \| \psi_{\alpha c} \|^2$, and there is a set $S \subseteq \partial \D$ with $\mu_1(\partial \D \backslash S) = \mu_2(S)  = 0$.
\newline

Let $\psi_1 = \chi_S(U) \psi_{\alpha c}$, $\psi_2 = \psi - \psi_1$.  By the spectral theorem, we may observe that
\begin{align*}
\mu_{\psi_1}(E)
  & = \langle \psi_1, \chi_E(U) \psi_1 \rangle \\
  & = \langle \psi_{\alpha c}, \chi_{E \cap S} \psi_{\alpha c} \rangle \\
  & = \mu_{\psi_{\alpha c}}(E \cap S) \\
  & = \mu_1(E).
\end{align*}

Thus, $\mu_{\psi_1} = \mu_1$.  In particular, $\mu_{\psi_1}$ is U$\alpha$H.
\newline

Define the projection onto a ball of radius $N$ via $P_N x = \sum_{|n| \leq N} \langle \varphi_n, x \rangle \varphi_n$.  We may choose a constant $C$ which depends solely on $d$ such that $\#\{ n : |n| \leq N \} \leq CN^d $.  In particular, $\|P_N\|_1 = \tr(P_N) \leq CN^d$.  By Proposition~\ref{rage}, we may choose a constant $C = C_{\psi_1}$ for which
\[
\frac{1}{K} \sum_{j=0}^{K-1} \langle \psi_(j), P_N \psi_1(j) \rangle < C_{\psi_1} \| P_N \|_1 K^{-\alpha}.
\]
Using the fact that $P_N$ is a projection, we see that
\begin{align*}
\frac{1}{K} \sum_{k=0}^{K-1} \| P_N \psi_1(k) \|^2
  & = \frac{1}{K} \sum_{k=0}^{K-1} \langle \psi_1(k), P_N \psi_1(k) \rangle \\
  & < C_{\psi_1} \| P_N \|_1 K^{-\alpha} \\
  & \leq C_1 N^d K^{-\alpha}
\end{align*}
with $C_1 = C_{\psi_1} C$.

Thus, we have
\begin{align*}
\tilde{P}_{\inn}(N,K)
   & = \frac{1}{K} \sum_{k=0}^{K-1} \|P_N \psi(k) \|^2 \\
   & \leq \frac{1}{K} \sum_{k=0}^{K-1} \big( \|P_N \psi_1(k) \| + \| P_N \psi_2(k) \| \big)^2 \\
   & \leq \frac{1}{K} \sum_{k=0}^{K-1} \big( \|P_N \psi_1(k) \| + \|\psi_2 \| \big)^2 \\
   & \leq \left( \sqrt{ \frac{1}{K}\sum_{k=0}^{K-1} \|P_N \psi_1(k) \|^2 } + \|\psi_2\| \right)^2 \\
   & < \left( \sqrt{ C_1 N^d K^{-\alpha} } + \| \psi_2\| \right)^2,
\end{align*}
where we have used projectivity of $P_N$ and unitarity of $U$ in the third line and Cauchy-Schwarz in the fourth.  Choose $\eta>0$ with $ \eta < \frac{\sqrt{6}-2}{2} $.  Hence,
\begin{align*}
\tilde{P}_{\inn} \left( \left( \frac{\eta^2 \|\psi_1\|^4 K^{\alpha}}{ C_1} \right)^{1/d}, K \right)
 & < \left( \eta \| \psi_1 \|^2  + \| \psi_2 \|\right)^2 \\
 & < 1 - \frac{1}{2} \| \psi_1 \|^2,
\end{align*}
where we have used $\|\psi_1\| , \|\psi_2\| \leq 1$ and the upper bound on $\eta$ to obtain the second line of the estimate.  It follows that
\[
\tilde{P}_{\out}\left( \left( \frac{\eta^2 \|\psi_1\|^4 K^{\alpha}}{C_1} \right)^{1/d}, K \right) > \frac{1}{2} \| \psi_1 \|^2.
\]
Making use of Remark~\ref{r.moments.pout}, we see that
\begin{align*}
\left\langle |X|^p_{\psi} \right\rangle(K)
   & \geq  \left( \frac{\eta^2 \|\psi_1\|^4 K^{\alpha}}{C_1} \right)^{p/d} \tilde{P}_{\out} \left( \left( \frac{\eta^2 \|\psi_1\|^4 K^{\alpha}}{C_1} \right)^{1/d} , K \right) \\
   & \geq C_{\psi,p} K^{\frac{\alpha p}{d}},
\end{align*}
where we take $C_{\psi,p} =  \frac{\| \psi_1 \|^2}{2} \left( \frac{ \eta^2 \| \psi_1 \|^4 }{C_1} \right)^{p/d}$.
\end{proof}

\begin{remark}
One should note that Proposition~\ref{ac.not.zero} is indeed stronger than the second part of Proposition~\ref{UaH}, since $\mu_{\psi} $ U$\alpha$H $\implies \mu_{\psi}$ is $\alpha$-continuous $\implies P_{\alpha c} \psi = \psi \neq 0$.
\end{remark}

Proposition~\ref{ac.not.zero} immediately yields a lower bound on the transport exponents $\tilde{\beta}_{\psi}^{\pm}(p)$.

\begin{coro}\label{c.coroone}
Suppose that $\psi$ is such that $P_{\alpha c} \psi \neq 0$ for some $0 \leq \alpha \leq 1$. One then has
\[
\tilde{\beta}_{\psi}^-(p) \geq \frac{\alpha}{d}.
\]
\end{coro}

\begin{proof}
This is immediate from the definitions.
\end{proof}

A succinct way of restating this last result involves the concept of the (upper) Hausdorff dimension of a measure. Recall the following definition; see \cite{F} for background and more information.

\begin{defi}
Let $\mu$ be a finite Borel measure on $\partial \D$. The upper Hausdorff dimension of $\mu$ is given by
$$
\dim_{\Hd}^+(\mu) = \inf \{ \dim_{\Hd}(S) : S \subset \partial \D \text{ measurable}, \; \mu(S) = \mu(\partial \D) \}.
$$
Loosely speaking, the upper Hausdorff dimension of the measure $\mu$ is the smallest Hausdorff dimension of a set which supports $\mu$.
\end{defi}

\begin{coro}\label{c.corotwo}
We have
\[
\tilde{\beta}_{\psi}^-(p) \geq \frac{\dim_{\Hd}^+(\mu_\psi^U)}{d}.
\]
\end{coro}

\begin{proof}
If $\dim_{\Hd}^+(\mu_\psi^U) = 0$, there is nothing to prove. Thus, let us assume that $\dim_{\Hd}^+(\mu_\psi^U) > 0$ and choose $\alpha \in (0,\dim_{\Hd}^+(\mu_\psi^U))$. Since $\alpha < \dim_{\Hd}^+(\mu_\psi^U)$, the definition of the upper Hausdorff dimension implies that in the Rogers-Taylor decomposition of $\mu_\psi^U$ into an $\alpha$-continuous piece and an $\alpha$-singular piece, the former must be nontrivial (for otherwise we could choose a suitable support of the latter to derive a contradiction). This implies that the previous corollary is applicable with the $\alpha$ in question and hence yields $\tilde{\beta}_{\psi}^{\pm}(p) \geq \frac{\alpha}{d}$. Since this estimate holds for every $\alpha \in (0,\dim_{\Hd}^+(\mu_\psi^U))$, the assertion follows.
\end{proof}

\subsection{Extracting the $\alpha$-Continuous Component of a Measure on the Circle}

The previous subsection has shown that a non-trivial $\alpha$-continuous component of a spectral measure leads to a corresponding dynamical lower bound. This motivates the question of finding a useful way to extract and study the $\alpha$-continuous component of a finite measure on $\partial \D$. In this subsection we summarize some known results that answer this question and provide a bridge to the discussion of CMV matrices in Section~\ref{s.4}.

Let $\mu$ be a finite measure on $\partial \D$. Given $\alpha \in (0,1)$ and $z_0 \in \partial \D$, let
$$
D^\alpha_\mu(z_0) = \limsup_{\varepsilon \downarrow 0} \frac{\mu \{ z \in \partial \D : z = z_0 e^{i \varphi}, \, \varphi \in (-\varepsilon,\varepsilon) \}}{(2 \varepsilon)^\alpha} \in [0,\infty]
$$
and
$$
S_\alpha = \{ z \in \partial \D : D^\alpha_\mu(z) = \infty \}.
$$

The following result is due to Rogers and Taylor \cite{RT1, RT2}; see also \cite[Theorem~10.8.7]{S2} and its discussion therein.

\begin{theorem}\label{t.rogerstaylor}
Consider the restrictions
$$
\mu_{\alpha c} = \mu \Big|_{\partial \D \setminus S_\alpha} , \quad \mu_{\alpha s} = \mu \Big|_{S_\alpha}.
$$
Then, $\mu_{\alpha c}$ gives zero weight to measurable $S \subseteq \partial \D$ with $h^\alpha(S) = 0$ and $\mu_{\alpha s}$ is supported by a measurable set $S \subseteq \partial \D$ with $h^\alpha(S) = 0$. In particular,
$$
\mu = \mu_{\alpha c} + \mu_{\alpha s}
$$
is the decomposition of $\mu$ into an $\alpha$-continuous piece and an $\alpha$-singular piece.
\end{theorem}

This shows that the $\alpha$-derivative $D^\alpha_\mu$ of a measure may be used to extract the $\alpha$-continuous component of it. The following connection is also very useful. Recall that the Carath\'eodory function $F$ associated with $\mu$ is given by
$$
F(z) = \int_{\partial \D} \frac{e^{i\theta} + z}{e^{i\theta} - z} \, d\mu(e^{i\theta}).
$$

The following equivalence is established in \cite[Lemma~10.8.6]{S2}.

\begin{prop}\label{p.djlsconnection}
For $z_0 \in \partial \D$, we have
$$
D^\alpha_\mu(z_0) = \infty \; \Leftrightarrow \; \limsup_{r \uparrow 1} (1 - r)^{1-\alpha} |F(rz_0)| = \infty.
$$
\end{prop}

As we will see in Section~\ref{s.4}, the rate of divergence of $|F(rz_0)|$ as $r \uparrow 1$ can be studied by quite effective means in the case of CMV matrices. In particular, this provides a direct path toward dynamical lower bounds for such operators.

\subsection{A Consequence of the Parseval Identity}\label{ss.parseval}

In this subsection we work out a unitary analog of a lemma that has proved to be useful in the self-adjoint case. Namely, it is a consequence of the Parseval identity that a modified time average of the dynamics is related, via Fourier transform, to an energy average of the resolvent of the operator.

\begin{lemma}\label{l.parseval}
For $K \ge 1$ and $n$ arbitrary, we have
$$
\sum_{k \ge 0} e^{-2k/K} a(n,k) = e^{2/K} \int_0^{2\pi} \left| \langle \varphi_n , (U - e^{1/K + i\theta})^{-1} \psi \rangle \right|^2 \, \frac{d\theta}{2\pi}.
$$
\end{lemma}

\begin{proof}
The proof is an adaptation of the proof of \cite[Lemma~3.2]{KKL} to the unitary case at hand. We give the details for the convenience of the reader.

Denote
$$
f(k) = \begin{cases} e^{-k/K} \int_{\partial \D} z^k \overline{u_\psi^n(z)} \, d\mu_\psi(z) & k \ge 0, \\ 0 & k < 0.  \end{cases}
$$
Then,
\begin{align*}
\hat f(\theta) & = \sum_{k \in \Z} e^{- ik\theta} f(k) \\
& = \sum_{k \ge 0} e^{-ik\theta} e^{-k/K} \int_{\partial \D} z^k \overline{u_\psi^n(z)} \, d\mu_\psi(z) \\
& = \int_{\partial \D} \frac{\overline{u_\psi^n(z)} \, d\mu_\psi(z)}{1 - e^{-1/K - i\theta} z} \\
& = e^{1/K + i\theta} \int_{\partial \D} \frac{\overline{u_\psi^n(z)} \, d\mu_\psi(z)}{e^{1/K + i\theta} - z} \\
& = - e^{1/K + i\theta} \langle \varphi_n , (U - e^{1/K + i\theta})^{-1} \psi \rangle.
\end{align*}
The Parseval identity now implies that
\begin{align*}
\sum_{k \ge 0} e^{-2k/K} a(n,k) & = \|f\|^2_{\ell^2(\Z)} \\
& = \frac{1}{2\pi} \|\hat f\|^2_{L^2(0,2\pi)} \\
& = e^{2/K} \int_0^{2\pi} \left| \langle \varphi_n , (U - e^{1/K + i\theta})^{-1} \psi \rangle \right|^2 \, \frac{d\theta}{2\pi},
\end{align*}
as claimed.
\end{proof}

\begin{remark} \label{exp.averages}
Lemma~\ref{l.parseval} suggests that instead of Ces\`aro averages, we consider the following averages,
$$
\langle f \rangle (K) = \frac{2}{K} \sum_{k \ge 0} e^{-2k/K} f(j).
$$
The transport exponents associated with these time averages are actually the same as the ones associated with Ces\`aro time averages, provided that the function $f$ satisfies some power-law upper bound; see \cite[Lemma~2.19]{DT}. Thus, when studying the quantities $\tilde \beta_\psi^\pm(p)$, we can freely use the more convenient underlying time average.
\end{remark}

\subsection{Packing Dimensions of Spectral Measures}

In \cite{GSB}, the authors prove a companion result to those in \cite{L}, by bounding $\tilde{\beta}^+$ from below by the packing dimension of the relevant spectral measure in the self-adjoint case.

Proposition~1 of \cite{GSB} carries over to the present context. For the reader's convenience, we produce the details presently.

\begin{lemma} \label{gsb1}
Let $F \subseteq A$ be a subset of the indexing set of the orthonormal basis $(\varphi_n)_{n \in A}$.  Given $K \in \Z_+$ and $0 < \epsilon < 1$, choose $N = N(K,\epsilon)$ so that
\[
2^{N-2} \leq \frac{K \pi}{\sqrt{\epsilon}} < 2^{N-1}.
\]

Partition $\partial \D$ into dyadic arcs as follows: for $ 0 \leq j \leq N$, put
\begin{align*}
\theta_{j,N} & = \frac{j\pi}{2^{N-1}}, \\
\gamma_{j,N} & = e^{i \theta_{j,N}}, \\
\Gamma_{j,N} & = \left\{ e^{i \theta} : \theta_{j,N} \leq \theta < \theta_{j+1,N} \right\}.
\end{align*}

Given $\epsilon > 0$, one has
\[
\frac{1}{K} \sum_{n \in F} \sum_{l = 0}^{K-1} \left| \left\langle \varphi_n , \psi(l) \right\rangle \right|^2 \leq 2 \epsilon + \frac{8 \pi}{\sqrt{\epsilon}} \sum_{n \in F} \sum_{j = 0}^{2^N-1} \left| \left\langle \varphi_n, \chi_{_{\Gamma_{j,N}}} (U) \psi \right\rangle \right|^2.
\]
\end{lemma}

\begin{proof}
We can see that $\partial \D$ is the disjoint union of the $\Gamma_{j,N}$ as $j$ runs from 0 to $2^N - 1$. We may then approximate $\psi(l) = U^l \psi$ by
\[
\psi_K(l) = \sum_{j=0}^{2^N-1} \gamma_{j,N}^l \chi_{_{\Gamma_{j,N}}}(U) \psi.
\]
Indeed, one readily observes that, for $0 \leq l \leq K$, one has
\begin{align*}
\left\| \psi(l) - \psi_K(l) \right\|^2
   & = \left\| U^l \psi - \sum_{j=0}^{2^N-1} \gamma_{j,N}^l \chi_{_{\Gamma_{j,N}}}(U) \psi\right\|^2 \\
   & = \sum_{j=0}^{2^N-1} \int_{\Gamma_{j,N}} \! \left| z^l - \gamma_{j,N}^l \right|^2 \, \mathrm d \mu_{\psi}(z) \\
   & \leq \sum_{j=0}^{2^N-1} \int_{\Gamma_{j,N}} \! \left( \frac{l\pi}{2^{N-1}} \right)^2 \, \mathrm d \mu_{\psi}(z) \\
   & = \left( \frac{l\pi}{2^{N-1}} \right)^2 \\
   & < \epsilon.
\end{align*}
The first line is a definition, the second follows from the spectral theorem, the third by construction of the $\Gamma_{j,N}$, the fourth from $\mu_{\psi}(\partial \D) = \|\psi\| = 1$, and the fifth from our choice of $N$ and $0 \leq l \leq K$. It follows that
\begin{align*}
   & \frac{1}{K} \sum_{n \in F} \sum_{l = 0}^{K-1} \left| \left\langle \varphi_n , \psi(l) \right\rangle \right|^2 \\
   \leq & \frac{2}{K} \sum_{n \in F} \sum_{l = 0}^{K-1} \left| \left\langle \varphi_n , \psi_K(l) - \psi(l) \right\rangle \right|^2 + \frac{2}{K} \sum_{n \in F} \sum_{l = 0}^{K-1} \left| \left\langle \varphi_n , \psi_K(l) \right\rangle \right|^2 \\
   < & 2 \epsilon + \frac{2}{K} \sum_{n \in F} \sum_{l = 0}^{2^N-1} \left| \left\langle \varphi_n , \psi_K(l) \right\rangle \right|^2.
\end{align*}
We have used the elementary inequality $|a|^2 \leq 2 |a-b|^2 + 2|b|^2$ in the second line.  The third line is a consequence of previous estimates, nonnegativity of summands, and $K < 2^N$.  By expanding $\psi_K$ and performing some algebraic manipulations, we see that the above is equal to
$$
2 \epsilon + \frac{2}{K} \sum_{n \in F} \sum_{l = 0}^{2^N-1} \sum_{j=0}^{2^N-1} \sum_{k=0}^{2^N-1} \gamma_{j,N}^l \overline{\gamma_{k,N}^l} \left\langle \varphi_n ,  \chi_{_{\Gamma_{j,N}}} (U) \psi \right\rangle \overline{\left\langle \varphi_n , \chi_{_{\Gamma_{k,N}}} (U) \psi \right\rangle }.
$$
Summing over $l$ and $k$, this is equivalent to
\begin{align*}
    & 2 \epsilon + \frac{2}{K} \sum_{n \in F} \sum_{j=0}^{2^N-1} \sum_{k=0}^{2^N-1} 2^N \delta_{j,k} \left\langle \varphi_n ,  \chi_{_{\Gamma_{j,N}}} (U) \psi \right\rangle \overline{\left\langle \varphi_n ,  \chi_{_{\Gamma_{k,N}}} (U) \psi \right\rangle } \\
      = & 2 \epsilon + \frac{2}{K} \sum_{n \in F} \sum_{j=0}^{2^N-1}  2^N \left| \left\langle \varphi_n ,  \chi_{_{\Gamma_{j,N}}} (U) \psi \right\rangle \right|^2.
\end{align*}
By using the the relationship between $N,K$, and $\epsilon$, the above is at most
$$
        2 \epsilon + \frac{8 \pi}{\sqrt{\epsilon}} \sum_{n \in F} \sum_{j=0}^{2^N-1}  \left| \left\langle \varphi_n ,  \chi_{_{\Gamma_{j,N}}} (U) \psi \right\rangle \right|^2,
$$
which completes the proof of the lemma.
\end{proof}

\begin{prop} \label{gsb2}
Choose $d$ as in Propositions~\ref{UaH} and \ref{ac.not.zero}. Given $N \in \Z_+$ and $0 < \alpha < 1$, let $I_{N,\alpha} = \{ j : \mu( \Gamma_{j,N}) < 2^{-N\alpha} \}$, $A_{N,\alpha} = \bigcup_{j \in I_{N,\alpha}} \Gamma_{j,N}$, and $b_{N,\alpha} = \mu(A_{N,\alpha})$.  If $b_{N,\alpha} > 0$, then there exists a constant $M_{\alpha,d}$ depending only on $\alpha$ and $d$ such that for all $K$ with $ b_{N,\alpha} 2^{N-2} \leq 9 \pi K < b_{N,\alpha} 2^{N-1} $, one has
\[
\tilde{P}_{\out} \left( M_{\alpha,d} \left( b_{N,\alpha}^{3-\alpha} K^{\alpha} \right)^{1/d} ,K \right) \geq \frac{b_{N,\alpha}}{2}.
\]
\end{prop}

\begin{proof}
Put $\psi_N = \chi_{A_{N,\alpha}}(U) \psi$. Evidently, we have
$$
\| \psi_N \|^2 = \langle \psi, \chi_{A_{N,\alpha}}(U) \psi \rangle = \mu(A_{N,\alpha}) = b_{N,\alpha}.
$$

Pick $\eta > 0$ with $ \eta < \frac{\sqrt{6}-2}{4} $ and put $\epsilon = \left( \eta b_{N,\alpha} \right)^2$.  Given $m \in \Z_+$, take $F_m = \{ n : |n| \leq m \}$. We note then that $ b_{N,\alpha} 2^{N-2} \leq 9 \pi K < b_{N,\alpha} 2^{N-1} $ is equivalent to $ 2^{N-2} \leq \frac{K \pi}{\sqrt{\epsilon}} < 2^{N-1} $. Thus, applying Lemma \ref{gsb1} to $\epsilon, F_m$ and $\psi_N$, we see that
$$
\frac{1}{K} \sum_{|n| \leq m} \sum_{l = 0}^{K-1} \left| \left\langle \varphi_n , U^l\psi_N \right\rangle \right|^2 \leq 2 \eta^2 b_{N,\alpha}^2 + \frac{8 \pi}{\eta b_{N,\alpha}} \sum_{|n| \leq m} \sum_{j = 0}^{2^N-1} \left| \left\langle \varphi_n, \chi_{_{\Gamma_{j,N}}} (U) \psi_N \right\rangle \right|^2.
$$

We can control the sum on the right hand side as follows:
\begin{align*}
\sum_{|n| \leq m} \sum_{j = 0}^{2^N-1} \left| \left\langle \varphi_n, \chi_{_{\Gamma_{j,N}}} (U) \psi_N \right\rangle \right|^2
   & = \sum_{|n| \leq m} \sum_{j \in I_{N,\alpha}} \left| \left\langle \varphi_n, \chi_{_{\Gamma_{j,N}}} (U) \psi \right\rangle \right|^2 \\
   & \leq \sum_{|n| \leq m} \sum_{j \in I_{N,\alpha}} \left\| \chi_{_{\Gamma_{j,N}}} (U) \varphi_n \right\|^2 \left\| \chi_{_{\Gamma_{j,N}}} (U) \psi \right\|^2  \\
   & < \sum_{|n| \leq m} \sum_{j \in I_{N,\alpha}} 2^{-N\alpha} \left\| \chi_{_{\Gamma_{j,N}}} (U) \varphi_n \right\|^2 \\
   & \leq \sum_{|n| \leq m} 2^{-N\alpha} \\
   & \leq C_d m^d 2^{-N\alpha}.
\end{align*}
The first line holds because $\psi_N = \chi_{A_{N,\alpha}}(U) \psi$.  The second line follows from Cauchy-Schwarz, the third by definition of $I_{N,\alpha}$, and the fourth from $\| \varphi_n \| = 1$.  In the fifth line, $C_d$ is a constant which only depends on $d$.

Now, take
$$
m = \left( \frac{1}{4 \pi C_d} \left( \eta b_{N,\alpha} \right)^3 2^{N \alpha} \right)^{1/d}.
$$
Substituting this value of $m$ into the above inequality yields
\begin{align*}
\frac{1}{K} \sum_{|n| \leq m} \sum_{l = 0}^{K-1} \left| \left\langle \varphi_n , U^l\psi_N \right\rangle \right|^2
& \leq 2  \eta^2 b_{N,\alpha}^2 + \frac{8 \pi}{\eta b_{N,\alpha}} C_d m^d 2^{-N\alpha} \\
 & = \left( 2 \eta b_{N,\alpha} \right)^2.
 \end{align*}

Now, let $P_m$ be the projection onto a ball of radius $m$, that is, $P_m = \sum_{|n| \leq m} \langle \cdot, \varphi_n \rangle \varphi_n$. With $\psi_N' = \psi - \psi_N$, we have
\begin{align*}
\tilde{P}_{\inn}(m,K)
  & = \frac{1}{K} \sum_{l = 0}^{K-1} \| P_m \psi(l) \|^2 \\
  & = \frac{1}{K} \sum_{l = 0}^{K-1} \| P_m U^l (\psi_N + \psi_N') \|^2 \\
  & \leq \frac{1}{K} \sum_{l = 0}^{K-1} \left( \| P_m U^l \psi_N  \|^2 + 2 \| P_m U^l \psi_N  \| \|P_m U^l \psi_N ' \| + \|P_m U^l \psi_N ' \|^2 \right) \\
  & \leq \left( 2 \eta b_{N,\alpha} \right)^2 + 4 \eta b_{N,\alpha} + \| \psi_N ' \|^2.
\end{align*}
In the final line, we have used the previous estimate, projectivity of $ P_m $, unitarity of $U$, $\| \psi_N' \| \leq 1$, and Cauchy-Schwarz.  Thus, we see that
\begin{align*}
\tilde{P}_{\out}(m,K)
   & = 1 - \tilde{P}_{\inn}(m,K) \\
   & \geq 1 - \left( 2 \eta b_{N,\alpha} \right)^2 - 4 \eta b_{N,\alpha} - \| \psi_N ' \|^2 \\
   & = b_{N,\alpha} - \left( 2 \eta b_{N,\alpha} \right)^2 - 4 \eta b_{N,\alpha} \\
   & > \frac{b_{N,\alpha}}{2}.
\end{align*}
The first line is trivial, the second follows from the estimate above, the third from orthogonality of $\psi_N $ and $\psi_N'$, and the final follows from $b_{N,\alpha} \leq 1$ and the bound on $\eta$. In particular, we may deduce that $\tilde{P}_{\out}(R,K) \geq b_{N,\alpha}/2$ whenever $R \leq m $. Recalling the  relationships between the variables, we have $ 2^N > \frac{18 \pi K}{b_{N,\alpha}} $, which yields
$$
m = \left( \frac{1}{4 \pi C_d} \left( \eta b_{N,\alpha} \right)^3 2^{N \alpha} \right)^{1/d}
>
 M_{\alpha,d} \left( K^{\alpha} b_{N,\alpha}^{3-\alpha} \right)^{1/d},
$$
with
$$
M_{\alpha,d} = \left( \frac{\eta^3 (18\pi)^{\alpha}}{4 \pi C_d} \right)^{1/d}.
$$
The proposition follows.
\end{proof}

Recall the definition of the $\alpha$-dimensional packing measure $p^\alpha$; compare \cite{F}.

\begin{defi}
Suppose $S \subseteq \partial \D$ and $\delta > 0$. A $\delta$-packing with centers in $S$ is a countable collection of mutually disjoint closed arcs, $\{I_j\}_{j \in \Z_+}$, each of which has length bounded by $\delta$ and center belonging to $S$. We set
$$
p^\alpha_\delta(S) = \sup \left\{ \sum_{j = 1}^\infty |I_j|^\alpha : \{I_j\}_{j \in \Z_+} \text{ is a $\delta$-packing with centers in } S \right\}
$$
and
$$
\tilde p^\alpha (S) = \lim_{\delta \to 0} p^\alpha_\delta(S) = \inf_{\delta > 0} p^\alpha_\delta(S).
$$
We also set
$$
p^\alpha(S) = \inf \left\{ \sum_{k=1}^\infty \tilde p^\alpha(S_k) : S = \bigcup_{k=1}^\infty S_k \right\}.
$$
\end{defi}

Note that $\tilde p^\alpha_\delta(S)$ decreases as $\delta$ decreases. This shows that the limit and the infimum above are indeed equal. Restricted to Borel sets $S$, $p^\alpha$ is a Borel measure. It is not hard to show that $p^\alpha(S) = 0$ whenever $p^{\alpha '} (S) < \infty$ and $0 \le \alpha' < \alpha$.

\begin{defi}
The packing dimension of a non-empty subset $S$ of $\partial \D$ is given by
\begin{align*}
\dim_{\Pa}(S) & = \sup \{ \alpha : p^\alpha(S) > 0 \} = \sup \{ \alpha : p^\alpha(S) = \infty \} \\ & = \inf \{ \alpha : p^\alpha(S) < \infty \} = \inf \{ \alpha : p^\alpha(S) = 0 \}.
\end{align*}
\end{defi}

\begin{defi}
Let $\mu$ be a finite Borel measure on $\partial \D$. The upper packing dimension of $\mu$ is given by
$$
\dim_{\Pa}^+(\mu) = \inf \{ \dim_{\Pa}(S) : S \subset \partial \D \text{ measurable}, \; \mu(S) = \mu(\partial \D) \}.
$$
\end{defi}

For the proof of the corollary below, the following characterization of the upper packing dimension is useful; compare Chapter~10 of \cite{F} and the appendix of \cite{GSB}.

\begin{prop}\label{p.packdim.char}
The upper packing dimension of $\mu$ is also given by
$$
\dim_{\mathrm{P}}^+(\mu)
=
\mu \esssup_{E \in \R}
\left(
 \limsup_{\varepsilon \to 0}
\frac{\log (\mu([E-\varepsilon,E+\varepsilon]))}{\log(\varepsilon)}
\right).
$$
\end{prop}

\begin{coro}
We have
\[
\tilde{\beta}_{\psi}^+(p) \geq \frac{\dim_{\Pa}^+(\mu_\psi^U)}{d}.
\]
\end{coro}

\begin{proof}
We argue as in \cite{GSB}. If $\dim_{\Pa}^+( \mu_{\psi}^U) = 0$, then there is nothing to prove, so assume given $0 \leq \alpha < \dim_{\Pa}^+( \mu_{\psi}^U) $. One easily verifies that
$$
E \in \limsup_{N \to \infty} A_{N,\alpha}
\iff
 \limsup_{\varepsilon \to 0}
\frac{\log (\mu([E-\varepsilon,E+\varepsilon]))}{\log(\varepsilon)}
>
\alpha.
$$
In particular, Proposition \ref{p.packdim.char} implies that $\mu \left( \limsup_{N \to \infty} A_{N,\alpha} \right) > 0$, which, by Borel-Cantelli, implies that $\sum_{N \in \Z_+} \mu( A_{N,\alpha} ) = \infty$.  Hence, there is a sequence $(N_j)_{j=1}^{\infty}$ of integers such that $b_{N_j} = \mu(A_{N_j,\alpha}) > N_j^{-2}$ (for if not, a simple comparison would imply that the divergent sum above converges). Lemma \ref{gsb2} then implies that
$$
\tilde{P}_{\out} \left( M_{\alpha,d} \left( b_{N_j}^{3-\alpha} K_j^{\alpha} \right)^{1/d} , K_j \right) \geq \frac{N_j^{-2}}{2}.
$$

Of course, we have chosen an increasing subsequence of sampling times $K_1 < K_2 < \cdots$ so that Proposition~\ref{gsb2} is relevant, that is, such that $ b_{N_j} 2^{N_j - 2} \leq 9 \pi K_j < b_{N_j} 2^{N_j-1}$. We can then make use of Remark~\ref{r.moments.pout} to see that
\begin{align*}
\left\langle |X|^p_{\psi} \right\rangle(K)
   & \geq M_{\alpha,d}^p \left( b_{N_j}^{3-\alpha} K_j^{\alpha} \right)^{p/d} \frac{N_j^{-2}}{2} \\
   & \geq \frac{1}{2} M_{\alpha,d}^p K_j^{\alpha p/d}  N_j^{(2\alpha-6)\frac{p}{d}-2}.
\end{align*}

We claim that $\tilde{\beta}_\psi^+(p) \geq \alpha/d$.  As a consequence of the above inequality, it suffices to prove that
$$
\lim_{j \to \infty} \frac{\log(N_j)}{ \log(K_j)} = 0.
$$
To that end, we begin by noticing that $ \frac{\log(N_j)}{\log(K_j)} $ is uniformly bounded above.  To see this, simply choose $j$ large enough that $ 2^{N_j} \geq 36 \pi N_j^3 $, and observe that one has $ K_j \geq \frac{b_{N_j,\alpha}}{36 \pi} 2^{N_j} \geq N_j $ for such $j$ $\left( \text{by using } b_{N_j,\alpha} \geq N_j^{-2} \right)$.  Thus, $ \frac{\log(N_j)}{ \log(K_j)} \leq 1 $ for sufficiently large $j$, from which the boundedness observation follows.

Now, let $\epsilon > 0 $ be given, and choose $j$ sufficiently large so that $\log(N_j) < \epsilon N_j$. We can take the logarithm of the relationship $b_{N_j} 2^{N_j -2} \leq 9 \pi K_j$, use  $ b_{N_j} > N_{j}^{-2}$ and rearrange to obtain
$$
\frac{N_j \log(2)}{\log(K_j)} \leq 1 + \frac{\log(36 \pi) + 2 \log(N_j)}{\log(K_j)}.
$$
Evidently, the expression on the right is uniformly bounded for all $j$ by a positive constant, say, $C >0$.  Thus, for $j$ chosen sufficiently large as above, we have
\begin{align*}
0 & \leq \frac{\log(N_j)}{\log(K_j)} \\
  & \leq \frac{\epsilon N_j}{\log(K_j)} \\
  & \leq \frac{C \epsilon}{\log(2)}.
\end{align*}
Thus, we obtain $\tilde{\beta}_{\psi}^{+}(p) \geq \frac{\alpha}{d}$. Since this holds for all $\alpha < \dim_{\Pa}^+(\mu_{\psi})$, the proposition follows.
\end{proof}

\section{CMV Matrices}\label{s.4}

In this section we consider the special case where the unitary operator is given by a CMV matrix. A CMV matrix is a semi-infinite matrix of the form
$$
\mathcal{C} = \begin{pmatrix}
{}& \bar\alpha_0 & \bar\alpha_1 \rho_0 & \rho_1
\rho_0
& 0 & 0 & \dots & {} \\
{}& \rho_0 & -\bar\alpha_1 \alpha_0 & -\rho_1
\alpha_0
& 0 & 0 & \dots & {} \\
{}& 0 & \bar\alpha_2 \rho_1 & -\bar\alpha_2 \alpha_1 &
\bar\alpha_3 \rho_2 & \rho_3 \rho_2 & \dots & {} \\
{}& 0 & \rho_2 \rho_1 & -\rho_2 \alpha_1 &
-\bar\alpha_3
\alpha_2 & -\rho_3 \alpha_2 & \dots & {} \\
{}& 0 & 0 & 0 & \bar\alpha_4 \rho_3 & -\bar\alpha_4
\alpha_3
& \dots & {} \\
{}& \dots & \dots & \dots & \dots & \dots & \dots & {}
\end{pmatrix},
$$
where $\alpha_n \in \D = \{ w \in \C : |w| < 1 \}$ and $\rho_n = (1-|\alpha_n|^2)^{1/2}$. $\mathcal{C}$ defines a unitary operator on $\ell^2(\Z_+)$.

These particular unitary operators play an important role in the theory of orthogonal polynomials on the unit circle as well as in the study of quantum walks in one dimension. Moreover they provide a canonical representation of general unitary operators in the following sense. Given any unitary operator $U$ in $\mathcal{H}$ and an initial state $\psi \in \mathcal{H}$, the evolution $U^n \psi$ takes place inside the spectral subspace $\mathcal{H}_\psi$ generated by $U$ and $\psi$. The action of $U$ on $\mathcal{H}_\psi$ is unitarily equivalent to multiplication by $z$ in $L^2(\partial \D, d\mu_\psi^U)$. Choosing the so-called CMV basis of $L^2(\partial \D, d\mu_\psi^U)$, the matrix representation of the latter operator with respect to this basis is then given by a CMV matrix.

Sometimes it makes sense to consider extended CMV matrices, acting on $\ell^2(\Z)$. They have the exact same form, but are two-sided infinite and may be put in one-to-one correspondence with two-sided infinite sequences $\{ \alpha_n \}_{n \in \Z} \subset \D$. They are typically denoted by $\mathcal{E}$. Hence, such an extended CMV matrix, corresponding to a sequence $\{ \alpha_n \}_{n \in \Z} \subset \D$, takes the form
$$
\mathcal{E} = \begin{pmatrix}
{}& \dots & \dots & \dots & \dots & \dots & \dots & \dots & \dots & {} \\
{}& \dots & -\bar\alpha_{-3} \alpha_{-4} & -\rho_{-3} \alpha_{-4} & 0 & 0 & 0 & 0 & \dots & {} \\
{}& \dots & \bar\alpha_{-2} \rho_{-3} & -\bar\alpha_{-2} \alpha_{-3} & \bar\alpha_{-1} \rho_{-2} & \rho_{-1} \rho_{-2} & 0 & 0 & \dots & {} \\
{}& \dots & \rho_{-2} \rho_{-3} & -\rho_{-2} \alpha_{-3} & -\bar\alpha_{-1} \alpha_{-2} & -\rho_{-1} \alpha_{-2} & 0 & 0 & \dots & {} \\
{}& \dots & 0 & 0 & \bar\alpha_0 \rho_{-1} & -\bar\alpha_0 \alpha_{-1} & \bar\alpha_1 \rho_0 & \rho_1 \rho_0 & \dots & {} \\
{}& \dots & 0 & 0 & \rho_0 \rho_{-1} & -\rho_0 \alpha_{-1} & -\bar\alpha_1 \alpha_0 & -\rho_1 \alpha_0 & \dots & {} \\
{}& \dots & 0 & 0 & 0 & 0 & \bar\alpha_2 \rho_1 & -\bar\alpha_2 \alpha_1 & \dots & {} \\
{}& \dots & \dots & \dots & \dots & \dots & \dots & \dots & \dots & {}
\end{pmatrix}.
$$
For example, when the $\alpha_n$'s are obtained by sampling along the orbit of an invertible ergodic transformation, the general theory of such operators naturally considers the two-sided situation. Special cases of this scenario that are of great interest include the periodic case, the almost periodic case, and the random case.

\subsection{Spectral Regularity via Subordinacy Theory}\label{ss.subordinacy}

Here we describe conditions on a CMV matrix that imply that some of the dynamical results presented in the previous section are applicable. That is, we state criteria for local and global regularity of spectral measures that are effective in the sense that for any given CMV matrix, there is clear path toward establishing these sufficient conditions since they are phrased in terms of solution estimates, which can be obtained in a variety of ways from the coefficients of the given matrix.

We begin with the half-line case. Suppose a CMV matrix $\mathcal{C}$ with Verblunsky coefficients $\{ \alpha_n \}_{n \ge 0} \subset \D$ is given. The associated probability measure $\mu$ on the unit circle is given by the spectral measure associated with the unitary operator $\mathcal{C}$ on $\ell^2(\Z_+)$ and the unit vector $\delta_0 \in \ell^2(\Z_+)$. Recall that the Carath\'eodory function $F$ associated with $\mu$ is given by
$$
F(z) = \int_{\partial \D} \frac{e^{i\theta} + z}{e^{i\theta} - z} \, d\mu(e^{i\theta}).
$$

The transfer matrices associated with these Verblunsky coefficients are defined as follows. For $z \in \partial \D$, $\alpha \in \D$, and $\rho = (1 - |\alpha|^2)^{1/2}$, write
$$
T(z,\alpha) = \rho^{-1} \begin{pmatrix} z & - \bar{\alpha} \\ - \alpha z & 1 \end{pmatrix}.
$$
Then, for $n \ge 1$, let
$$
T_n(z) = T(z,\alpha_{n-1}) \cdots T(z,\alpha_{0}).
$$
We also set $T_0(z) = I$.


The orthonormal polynomials of the first and second kind are defined by
$$
\begin{pmatrix} \varphi_n(z) \\ \varphi_n^*(z) \end{pmatrix} = T_n(z) \begin{pmatrix} 1 \\ 1 \end{pmatrix}
$$
and
$$
\begin{pmatrix} \psi_n(z) \\ \psi_n^*(z) \end{pmatrix} = T_n(z) \begin{pmatrix} 1 \\ -1 \end{pmatrix},
$$
respectively; compare \cite[Proposition~3.2.1]{S1}.

For a sequence $a_0, a_1, a_2, \ldots$ of complex numbers and $L \in (0,\infty)$, let
$$
\|a\|_L = \sum_{n = 0}^{\lfloor L \rfloor} |a_n|^2 + (L - \lfloor L \rfloor) |a_{\lfloor L \rfloor + 1}|^2.
$$
That is, $\| \cdot \|_L$ is a local $\ell^2$ norm for integer values of $L$, and $\| \cdot \|^2_L$ is linearly interpolated in between.

The following result is \cite[Theorem~10.8.2]{S2}:

\begin{theorem}[OPUC Version of the Jitomirskaya-Last Inequality]\label{t.jitomirskayalast}
Suppose $z \in \partial \D$ and $r \in [0,1)$. Define $L(r)$ to be the unique solution of
$$
(1 - r) \|\varphi(z)\|_{L(r)} \|\psi(z)\|_{L(r)} = \sqrt{2}.
$$
Then, for some universal constant $A \in (1,\infty)$, we have
$$
A^{-1} \frac{\|\psi(z)\|_{L(r)}}{\|\varphi(z)\|_{L(r)}} \le |F(rz)| \le A \frac{\|\psi(z)\|_{L(r)}}{\|\varphi(z)\|_{L(r)}}.
$$
\end{theorem}

Recall that by Proposition~\ref{p.djlsconnection}, the divergence rate of $|F(rz)|$ is connected to the $\alpha$-derivative of $\mu$ at $z$. Combining this with Theorem~\ref{t.jitomirskayalast} one arrives at the following equivalence, which is \cite[Theorem~10.8.5]{S2}.

\begin{coro}\label{c.jitolast}
Given $\alpha \in (0,1)$, let $\beta = \frac{\alpha}{2-\alpha}$. Then, for $z_0 \in \partial \D$, we have
$$
D^\alpha_\mu(z_0) = \infty \; \Leftrightarrow \; \liminf_{L \to \infty} \frac{\|\varphi(z_0)\|_L}{\|\psi(z_0)\|_L^\beta} = 0.
$$
\end{coro}

This result has the following immediate consequence:

\begin{coro}\label{c.jitolast2}
Suppose that for $z_0 \in \partial \D$, we have
$$
\|\varphi(z_0)\|_L \gtrsim L^{\gamma_1} , \quad \|\psi(z_0)\|_L \lesssim L^{\gamma_2}
$$
for $L \ge 1$, where $0 < \gamma_1 < \gamma_2 < \infty$. Then, with
$$
\alpha = \frac{2\gamma_1}{\gamma_1 + \gamma_2},
$$
we have
$$
D^\alpha_\mu(z_0) < \infty.
$$
In particular, the restriction of $\mu$ to the set
$$
P(\gamma_1, \gamma_2) = \{ z \in \partial \D : \|\varphi(z)\|_L \gtrsim L^{\gamma_1} , \; \|\psi(z)\|_L \lesssim L^{\gamma_2} \}
$$
{\rm (}with implicit constants that may depend on $z${\rm )} is $\alpha$-continuous for this choice of $\alpha$.
\end{coro}

\begin{proof}
We have
$$
\frac{\|\varphi(z_0)\|_L^{2-\alpha}}{\|\psi(z_0)\|_L^\alpha} \gtrsim \frac{L^{\gamma_1(2-\alpha)}}{L^{\gamma_2 \alpha}} = L^{-\alpha(\gamma_1 + \gamma_2) + 2 \gamma_1} = L^0 = 1.
$$
This shows that
$$
\liminf_{L \to \infty} \frac{\|\varphi(z_0)\|_L^{2-\alpha}}{\|\psi(z_0)\|_L^\alpha} > 0,
$$
which in turn implies
$$
\liminf_{L \to \infty} \frac{\|\varphi(z_0)\|_L}{\|\psi(z_0)\|_L^\beta} > 0.
$$
The result therefore follows from Corollary~\ref{c.jitolast}.
\end{proof}

\bigskip

Let us now turn to the whole-line case and consider extended CMV matrices $\mathcal{E}$, determined by a two-sided infinite sequence of coefficients $\{ \alpha_n \}_{n \in \Z} \subset \D$. There is a close analog of Corollary~\ref{c.jitolast2}. Namely, suitable power-law estimates imply continuity properties of spectral measures. In fact, it suffices to have such power-law estimates on one half-line, say the right half-line for definiteness. However, these estimates need to hold ``uniformly in the boundary condition.'' That is, one has to consider all vector-valued sequences of the form
\begin{equation}\label{e.tmequ}
\begin{pmatrix} \xi_{n} \\ \zeta_{n} \end{pmatrix} = T_n(z) \begin{pmatrix} \xi_0 \\ \zeta_0 \end{pmatrix},
\end{equation}
where
\begin{equation}\label{e.tmequnorm}
|\xi_0| = |\zeta_0| = 1.
\end{equation}
The following result was shown in \cite{MO}. It is an adaptation of a result shown by Damanik, Killip, and Lenz in the Schr\"odinger context \cite{DKL}.

\begin{prop}\label{p.dklest}
Suppose that for $z \in \partial \D$, there are constants $0 < \gamma_1(z) < \gamma_2(z) < \infty$ and $0 < C_1(z), C_2(z) < \infty$ so that
$$
C_1(z) L^{\gamma_1(z)} \leq \|\xi\|_L \leq C_2(z) L^{\gamma_2(z)}, \quad L \ge 1
$$
for every solution of \eqref{e.tmequ} that is normalized in the sense of \eqref{e.tmequnorm}. Then, for every spectral measure $\mu$ of $\mathcal{E}$, we have $D^\alpha_\mu(z) < \infty$, where $\alpha = \frac{2\gamma_1(z)}{\gamma_1(z) + \gamma_2(z)}$.

In particular, if $S \subset \partial \D$ is a Borel set such that there are constants $0 < \gamma_1 < \gamma_2 < \infty$ and, for each $z \in S$, there are constants $0 < C_1(z), C_2(z) < \infty$ so that
$$
C_1(z) L^{\gamma_1} \leq \|\xi\|_L \leq C_2(z) L^{\gamma_2}, \quad L \ge 1
$$
for every $z \in S$ and every solution of \eqref{e.tmequ} that is normalized in the sense of \eqref{e.tmequnorm}, then the restriction of every spectral measure of $\mathcal{E}$ to $S$ is purely $\frac{2\gamma_1}{\gamma_1 + \gamma_2}$-continuous, that is, it gives zero weight to sets of zero $h^{\frac{2\gamma_1}{\gamma_1 + \gamma_2}}$ measure.
\end{prop}

\subsection{Quantum Walks on the Line}

Let us recall the standard quantum walk formalism. The Hilbert space is given by $\mathcal{H} = \ell^2(\Z) \otimes \C^2$. A basis is given by the elementary tensors $|n \rangle \otimes | \! \uparrow \rangle$, $|n \rangle \otimes | \! \downarrow \rangle$, $n \in \Z$. A time-homogeneous quantum walk scenario is given as soon as coins
\begin{equation}\label{e.timehomocoins}
C_{n} = \begin{pmatrix} c^{11}_{n} & c^{12}_{n} \\ c^{21}_{n} & c^{22}_{n} \end{pmatrix} \in U(2), \quad n \in \Z,
\end{equation}
are specified. As one passes from time $t$ to time $t+1$, the update rule of the quantum walk is as follows,
\begin{align}
|n \rangle \otimes | \! \uparrow \rangle & \mapsto c^{11}_{n} |n+1 \rangle \otimes | \! \uparrow \rangle + c^{21}_{n} |n-1 \rangle \otimes | \! \downarrow \rangle, \label{e.updaterule1} \\
|n \rangle \otimes | \! \downarrow \rangle & \mapsto c^{12}_{n} |n+1 \rangle \otimes | \! \uparrow \rangle + c^{22}_{n} |n-1 \rangle \otimes | \! \downarrow \rangle \label{e.updaterule2},
\end{align}
Extend this by linearity to general elements of $\mathcal{H}$. This defines a unitary operator $U$ on $\mathcal{H}$.

Order the basis of $\mathcal{H}$ as follows:
\begin{equation}\label{e.orderedbasis}
\ldots, |-1 \rangle \otimes | \! \uparrow \rangle , \, |-1 \rangle \otimes | \! \downarrow \rangle , \, |0 \rangle \otimes | \! \uparrow \rangle , \, |0 \rangle \otimes | \! \downarrow \rangle , \, |1 \rangle \otimes | \! \uparrow \rangle , \, |1 \rangle \otimes | \! \downarrow \rangle , \ldots .
\end{equation}
In this ordered basis, the matrix representation of $U : \mathcal{H} \to \mathcal{H}$ is given by
\begin{equation}\label{e.umatrixrep}
U = \begin{pmatrix}
{}& \dots & \dots & \dots & \dots & \dots & \dots & \dots & \dots & {} \\
{}& \dots & 0 & c^{12}_{-2} & 0 & 0 & 0 & 0 & \dots & {} \\
{}& \dots & c^{21}_{-1} & 0 & 0 & c^{11}_{-1} & 0 & 0 & \dots & {} \\
{}& \dots & c^{22}_{-1} & 0 & 0 & c^{12}_{-1} & 0 & 0 & \dots & {} \\
{}& \dots & 0 & 0 & c^{21}_{0} & 0 & 0 & c^{11}_{0} & \dots & {} \\
{}& \dots & 0 & 0 & c^{22}_{0} & 0 & 0 & c^{12}_{0} & \dots & {} \\
{}& \dots & 0 & 0 & 0 & 0 & c^{21}_{1} & 0 & \dots & {} \\
{}& \dots & \dots & \dots & \dots & \dots & \dots & \dots & \dots & {}
\end{pmatrix},
\end{equation}
as can be checked readily using the update rule \eqref{e.updaterule1}--\eqref{e.updaterule2}; compare \cite[Section~4]{CGMV}.\footnote{Note that we follow the conventions of \cite{CGMV} here. One could argue that the correct matrix to consider is the transpose of $U$ in \eqref{e.umatrixrep}. To conform with \cite{CGMV} and subsequent papers, we will consider the matrix $U$ as given above in what follows. In the Fibonacci example we discuss below, this does not make a difference since the matrix entries will be real and hence the transpose of $U$ is the inverse of $U$. Since our argument is based on spectral continuity of $U$, and the spectral continuity properties of $U$ and $U^{-1}$ are the same, the final result does not depend on the choice one makes at this juncture.}

Recall that an extended CMV matrix corresponding to Verblunsky coefficients $\{ \alpha_n \}_{n \in \Z}$ has the form
$$
\mathcal{E} = \begin{pmatrix}
{}& \dots & \dots & \dots & \dots & \dots & \dots & \dots & \dots & {} \\
{}& \dots & -\bar\alpha_{-3} \alpha_{-4} & -\rho_{-3} \alpha_{-4} & 0 & 0 & 0 & 0 & \dots & {} \\
{}& \dots & \bar\alpha_{-2} \rho_{-3} & -\bar\alpha_{-2} \alpha_{-3} & \bar\alpha_{-1} \rho_{-2} & \rho_{-1} \rho_{-2} & 0 & 0 & \dots & {} \\
{}& \dots & \rho_{-2} \rho_{-3} & -\rho_{-2} \alpha_{-3} & -\bar\alpha_{-1} \alpha_{-2} & -\rho_{-1} \alpha_{-2} & 0 & 0 & \dots & {} \\
{}& \dots & 0 & 0 & \bar\alpha_0 \rho_{-1} & -\bar\alpha_0 \alpha_{-1} & \bar\alpha_1 \rho_0 & \rho_1 \rho_0 & \dots & {} \\
{}& \dots & 0 & 0 & \rho_0 \rho_{-1} & -\rho_0 \alpha_{-1} & -\bar\alpha_1 \alpha_0 & -\rho_1 \alpha_0 & \dots & {} \\
{}& \dots & 0 & 0 & 0 & 0 & \bar\alpha_2 \rho_1 & -\bar\alpha_2 \alpha_1 & \dots & {} \\
{}& \dots & \dots & \dots & \dots & \dots & \dots & \dots & \dots & {}
\end{pmatrix}.
$$
In particular, if all Verblunsky coefficients with odd index vanish, the matrix becomes (recall that $\rho_n = (1-|\alpha_n|^2)^{1/2}$)
\begin{equation}\label{e.ecmvoddzero}
\mathcal{E} = \begin{pmatrix}
{}& \dots & \dots & \dots & \dots & \dots & \dots & \dots & \dots & {} \\
{}& \dots & 0 & - \alpha_{-4} & 0 & 0 & 0 & 0 & \dots & {} \\
{}& \dots & \bar\alpha_{-2} & 0 & 0 & \rho_{-2} & 0 & 0 & \dots & {} \\
{}& \dots & \rho_{-2} & 0 & 0 & -\alpha_{-2} & 0 & 0 & \dots & {} \\
{}& \dots & 0 & 0 & \bar\alpha_0 & 0 & 0 & \rho_0 & \dots & {} \\
{}& \dots & 0 & 0 & \rho_0 & 0 & 0 & -\alpha_0 & \dots & {} \\
{}& \dots & 0 & 0 & 0 & 0 & \bar\alpha_2 & 0 & \dots & {} \\
{}& \dots & \dots & \dots & \dots & \dots & \dots & \dots & \dots & {}
\end{pmatrix}.
\end{equation}

The matrix in \eqref{e.ecmvoddzero} strongly resembles the matrix representation of $U$ in \eqref{e.umatrixrep}. Note, however, that all $\rho_n$'s need to be real and non-negative for genuine CMV matrices, and this property is not guaranteed when matching \eqref{e.umatrixrep} and \eqref{e.ecmvoddzero}. But this can be easily resolved, as shown in \cite{CGMV}. Concretely, given $U$ as in \eqref{e.umatrixrep}, write
$$
c_n^{kk} = |c_n^{kk}| e^{i \sigma^k_n}, \quad n \in \Z, \; k \in \{ 1,2 \}, \; \sigma^k_n \in [0,2\pi)
$$
and define $\{ \lambda_n \}_{n \in \Z}$ by
$$
\lambda_0 = 1, \; \lambda_{-1} = 1, \; \lambda_{2n+2} = e^{-i \sigma^1_n} \lambda_{2n}, \; \lambda_{2n+1} = e^{i \sigma^2_n} \lambda_{2n-1}.
$$
With the unitary matrix $\Lambda = \mathrm{diag}(\ldots, \lambda_{-1} , \lambda_0 , \lambda_1 , \ldots)$, we then have
$$
\mathcal{E} = \Lambda^* U \Lambda,
$$
where $\mathcal{E}$ is the extended CMV matrix corresponding to the Verblunsky coefficients
\begin{equation}\label{e.correspondence}
\alpha_{2n+1} = 0 , \; \alpha_{2n} = \frac{\lambda_{2n}}{\lambda_{2n-1}} \bar c_n^{21}, \quad n \in \Z.
\end{equation}

In order for one to prove lower bounds for the spreading rates of a quantum walk on the line, the strategy is now clear. One needs to establish solution estimates for a given model that feed into Proposition~\ref{p.dklest}. Once Proposition~\ref{p.dklest} is shown to be applicable, its output provides the input for an application of Proposition~\ref{ac.not.zero} and its corollaries. In the next subsection we present a non-trivial example where this strategy may be implemented and yields lower bounds for the spreading rates of the quantum walk discussed there, which are explicit in terms of the parameters of the model.

\subsection{The Fibonacci Quantum Walk}

We discuss the special case of the Fibonacci quantum walk, which is an example that requires the full extent of the machinery developed in this paper in conjunction with the subordinacy result from \cite{MO} described in Subsection~\ref{ss.subordinacy}. In this example the sequence of coins takes only two different values, and the order in which these two unitary $2 \times 2$ matrices occur is determined by an element of the Fibonacci subshift.

Let us recall how the latter is generated. Consider two symbols, $a$ and $b$. The Fibonacci substitution $S$ sends $a$ to $ab$ and $b$ to $a$. This substitution rule can be extended by concatenation to finite and one-sided infinite words over the alphabet $\{a,b\}$. There is a unique one-sided infinite word that is invariant under $S$, denote it by $u$. It is, in an obvious sense, the limit as $n \to \infty$ of the words $s_n = S^n(a)$. That is, $s_0 = a$, $s_1 =ab$, $s_2= aba$, etc., so that $u = abaababaabaab \ldots$. The Fibonacci subshift $\Omega$ is given by
$$
\Omega = \{ \omega \in \{ a,b \}^\Z : \text{every finite subword of $\omega$ occurs in } u \}.
$$

Take $\theta_a , \theta_b \in (-\frac{\pi}{2}, \frac{\pi}{2})$ and consider the rotations
$$
C_a = \begin{pmatrix} \cos \theta_a & -\sin \theta_a \\ \sin \theta_a & \cos \theta_a \end{pmatrix}, \quad C_b = \begin{pmatrix} \cos \theta_b & -\sin \theta_b \\ \sin \theta_b & \cos \theta_b \end{pmatrix}.
$$
Given $\omega \in \Omega$, the associated sequence of coins $\{ C_{\omega,n} \}_{n \in  \Z}$ is given by $C_{\omega,n} = C_{\omega_n}$. The associated unitary operator will be denoted by $U_\omega$. Inspecting \eqref{e.correspondence} one sees that $U_\omega$ already has the form of an extended CMV matrix and we will therefore denote it by $\mathcal{E}_\omega$ to emphasize this fact.

Let us relabel the basis elements, ordered as in \eqref{e.orderedbasis}, and write them as $(\varphi_n)_{n \in \Z}$. We consider a non-zero finitely supported initial state $\psi \in \ell^2(\Z)$ and study the spreading in space of $\mathcal{E}_\omega^n \psi$ as $|n| \to \infty$ with respect to this basis.

Implementing the strategy outlined at the end of the previous subsection, \cite{DMY2} establishes solution estimates in the form needed in Proposition~\ref{p.dklest}. Thus, Proposition~\ref{ac.not.zero} and Corollaries~\ref{c.coroone} and \ref{c.corotwo} may be applied, and one obtains the following result:

\begin{theorem}\label{t.qwmain}
Define:
\begin{enumerate}

\item $I(z) = \Re(z)^2(\sec^2\theta_a + \sec^2\theta_b) + (\Re(z^2) \sec\theta_a\sec\theta_b - \tan\theta_a\tan\theta_b)^2 -2(\Re(z)^2 \sec^2\theta_a\sec^2\theta_b(\Re(z^2)-\sin\theta_a \sin\theta_b)) - 1$;

\item $C(z) = \max\{2+\sqrt{8+I(z)}, (\sec\theta_a)^{-1} , (\sec \theta_b)^{-1}\}$;

\item $\gamma_1(z) = \frac{\log \left(1+\frac{1}{4C(z)^2} \right)}{16 \log \phi}$, where $\phi$ is the golden mean;

\item $\gamma_2(z) = 4\log_2 K(z)$, where $K$ is a $z$-dependent constant;\footnote{The paper \cite{DMY2} contains an explicit expression for $K(z)$. Since this description is somewhat involved, we do not reproduce it here and refer the interested reader to \cite{DMY2}.}

\item $\beta(z) = \frac{2\gamma_1 (z)}{\gamma_1 (z) + 2\gamma_2 (z)+1}$.

\end{enumerate}
Then, for all $\psi, \omega, p$ as above, we have
$$
\tilde \beta^\pm_{\omega, \psi}(p) \ge \max\big\{ \beta(z) : z\in \mathrm{supp}\ \mu_{\mathcal{E}_\omega, \psi} \big\},
$$
where $\mu_{\mathcal{E}_\omega, \psi}$ denotes the spectral measure associated with the unitary operator $\mathcal{E}_\omega$ and the state $\psi$.
\end{theorem}

This theorem was stated and proved in \cite{DMY2}. The proof given there used Proposition~\ref{p.dklest}, proved in \cite{MO}, and Proposition~\ref{ac.not.zero} and Corollaries~\ref{c.coroone} and \ref{c.corotwo}, proved in this paper. Thus, these three papers work together in establishing Theorem~\ref{t.qwmain}.


\begin{thebibliography}{10}

\bibitem{BGVW} J.\ Bourgain, F.\ A.\ Gr\"unbaum, L.\ Vel\'asquez, J.\ Wilkening, Quantum recurrence of a subspace and operator-valued Schur functions, Preprint (arXiv:1302.7286).

\bibitem{CGMV} M.-J.\ Cantero, A.\ Gr\"unbaum, L.\ Moral, L.\ Vel\'azquez, Matrix-valued Szeg\H{o} polynomials and quantum random walks, \textit{Comm.\ Pure Appl.\ Math.}\ \textbf{63} (2010), 464--507.

\bibitem{CGMV2} M.-J.\ Cantero, A.\ Gr\"unbaum, L.\ Moral, L.\ Vel\'azquez, The CGMV method for quantum walks, \textit{Quantum Inf.\ Process.}\ \textbf{11} (2012), 1149--1192.

\bibitem{DKL} D.\ Damanik, R.\ Killip, D.\ Lenz, Uniform spectral properties of one-dimensional quasicrystals. III.~$\alpha$-continuity, \textit{Commun.\ Math.\ Phys.}\ \textbf{212} (2000), 191--204.

\bibitem{DMY2} D.\ Damanik, P.\ Munger, W.\ Yessen, Orthogonal polynomials on the unit circle with Fibonacci Verblunsky coefficients, II.~Applications, \textit{J.\ Stat.\ Phys.}\ \textbf{153} (2013), 339--362.

\bibitem{DT} D.\ Damanik, S.\ Tcheremchantsev, A general description of quantum dynamical spreading over an orthonormal basis and applications to Schr\"{o}dinger operators,  \textit{Discrete Contin.\ Dyn.\ Syst.}\ \textbf{28} (2010), 1381--1412.

\bibitem{F} K.\ Falconer, \textit{Techniques in Fractal Geometry}, John Wiley \& Sons, Ltd., Chichester, 1997.

\bibitem{GVWW} F.\ A.\ Gr\"unbaum, L.\ Vel\'azquez, A.\ H.\ Werner, R.\ F.\ Werner, Recurrence for discrete time unitary evolutions, \textit{Commun.\ Math.\ Phys.}\ \textbf{320} (2013), 543--569.

\bibitem{G1} I.\ Guarneri, Spectral properties of quantum diffusion on discrete lattices, \textit{Europhys.\ Lett.}\ \textbf{10} (1989), 95--100.

\bibitem{G2} I.\ Guarneri, On an estimate concerning quantum diffusion in the presence of a fractal spectrum, \textit{Europhys.\ Lett.}\ \textbf{21} (1993), 729--733.

\bibitem{GSB} I.\ Guarneri, H.\ Schulz-Baldes, Lower bounds on wave packet propagation by packing dimensions of spectral measures, \textit{Math.\ Phys.\ Electron.\ J.}\ \textbf{5} (1999), Paper~1, 16~pp.

\bibitem{J11} A.\ Joye, Random time-dependent quantum walks, \textit{Commun.\ Math.\ Phys.}\ \textbf{307} (2011), 65--100.

\bibitem{J12} A.\ Joye, Dynamical localization for d-dimensional random quantum walks, \textit{Quantum Inf.\ Process.}\ \textbf{11} (2012), 1251--1269.

\bibitem{JM} A.\ Joye, M.\ Merkli, Dynamical localization of quantum walks in random environments, \textit{J.\ Stat.\ Phys.}\ \textbf{140} (2010), 1025--1053.

\bibitem{KKL} R.\ Killip, A.\ Kiselev, Y.\ Last, Dynamical upper bounds on wavepacket spreading, \textit{Amer.\ J.\ Math.}\ \textbf{125} (2003), 1165--1198.

\bibitem{L} Y.\ Last, Quantum dynamics and decompositions of singular continuous spectra, \textit{J.\ Funct.\ Anal.}\ \textbf{142} (1996), 406--445.

\bibitem{MO} P.\ Munger, D.\ Ong, The H\"older continuity of spectral measures of an extended CMV matrix, Preprint (arXiv:1301.0501).

\bibitem{OS} C.\ De Oliveira, M.\ Simsen, A Floquet operator with purely point spectrum and energy instability, \textit{Ann.\ Henri Poincar\'e} \textbf{8} (2007), 1255--1277.

\bibitem{RT1} C.\ Rogers, S.\ Taylor, The analysis of additive set functions in Euclidean space, \textit{Acta Math.}\ \textbf{101} (1959), 273--302.

\bibitem{RT2} C.\ Rogers, S.\ Taylor, The analysis of additive set functions in Euclidean space. II, \textit{Acta Math.}\ \textbf{109} (1963), 207--240.

\bibitem{S1} B.\ Simon, \textit{Orthogonal Polynomials on the Unit Circle. Part~1. Classical Theory.}, Colloquium Publications, 54, American Mathematical Society, Providence (2005).

\bibitem{S2} B.\ Simon, \textit{Orthogonal Polynomials on the Unit Circle. Part~2. Spectral Theory}, Colloquium Publications, 54, American Mathematical Society, Providence (2005).

\bibitem{S} R.\ Strichartz, Fourier asymptotics of fractal measures, \textit{J.\ Funct.\ Anal.}\ \textbf{89} (1990), 154--187.

\end{thebibliography}
\end{document}